\documentclass[11pt]{amsart}

\usepackage{amsfonts}
\usepackage{amsmath}
\usepackage{amssymb}
\usepackage{amstext}
\usepackage{amsthm}
\usepackage{xspace}
\usepackage{graphicx}
\usepackage{psfrag}
\usepackage[all]{xypic}

\usepackage{fullpage}
\usepackage{pinlabel}
\newcommand{\R}{\mathbb R}

\newcommand{\Q}{\mathbb Q}
\newcommand{\Z}{\mathbb Z} 
\newcommand{\hf}{\widehat{HF}}
\newcommand{\spc}{{\rm Spin}^c}

\newcommand{\leg}{\mathcal{L}} 	%symbol for a legendrian knot
\newcommand{\knot}{K}  			%symbol for a knot
\newtheorem{theorem}{Theorem}[section]
\newtheorem{lemma}[theorem]{Lemma}
\newtheorem{cor}[theorem]{Corollary}
\newtheorem{prop}[theorem]{Proposition}
\newtheorem{conj}[theorem]{Conjecture}

\theoremstyle{definition}

\newtheorem{dfn}[theorem]{Definition}

\begin{document}
\title[Tight contact structures on $- \Sigma(2,3,6n-1)$]{Tight contact structures on the 
Brieskorn spheres $- \Sigma(2,3,6n-1)$ and contact invariants}

\author{Paolo Ghiggini}
\address{CNRS, Universit\'e de Nantes, FRANCE}
\email{paolo.ghiggini@univ-nantes.fr}
%\urladdr{http://www.math.sciences.univ-nantes.fr/~ghiggini}

\author{Jeremy Van Horn-Morris}
\address{American Institute of Mathematics, Palo Alto, CA}
\email{jvanhorn@aimath.org}
%\urladdr{http://aimath.org/\textasciitilde jvanhorn}
% \date{\today}

%\thanks{We would like to thank...}

\begin{abstract}We compute the Ozsv\'ath--Szab\'o contact invariants for all tight 
contact structures on the manifolds $- \Sigma(2,3,6n-1)$ using twisted coefficient
and a previous computation by the first author and Ko Honda. This computation 
completes the classification of the tight contact structures in this family of $3$-manifolds.
\end{abstract}
\maketitle

%%%%%%%%%%%%%%%%%%%%%%%%%%%%%%%%%%%%%%%%%%%%%%%%%%%%%%%%%%%%%%%%%%%%%%%%%%%%%%%%
\section{Introduction}
The family of $3$-manifolds $- \Sigma(2,3,6n-1)$ defined by the surgery diagram
in Figure \ref{topsurg.fig} has been an exciting playground for contact 
topologists for many years, and any progress in the knowledge of the tight contact 
structures in this family has led us to progress in our understanding of 
three-dimensional contact topology.

These manifolds first were used by Lisca and Mati\'c in \cite{lisca-matic:1} to give 
an example of the power of the recently discovered Seiberg--Witten invariants
in distinguishing tight contact structures. Later Etnyre and Honda 
\cite{etnyre-honda:1} proved that $-\Sigma (2,3,5)$ supports no tight
contact structure, giving the first example of such a manifold. Tight contact structures 
on $-\Sigma (2,3,17)$ were instrumental in the first vanishing theorem for the 
Ozsv\'ath--Szab\'o contact invariant in \cite{ghiggini:3}. Finally the first author proved 
in \cite{ghiggini:nonstein} that $- \Sigma(2,3,6n-1)$ carries a
strongly fillable contact structure which is not Stein fillable when 
$n \ge 3$, thus showing that strong and Stein fillability are different concepts in dimension
three. 

The goal of this paper is to give a complete classification of tight contact structures
on manifolds in this family, and to do that we will compute their Ozsv\'ath--Szab\'o 
contact invariants. The proof is a delicate computation using Heegaard Floer homology 
with twisted coefficients.

 \begin{figure} \centering
 \psfrag{0}{$0$}
 \psfrag{a}{$-n$}
 \includegraphics[width=4cm]{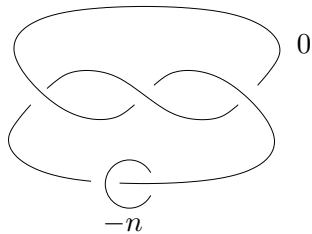}
 \caption{Surgery diagram for $- \Sigma(2,3,6n-1)$}
\label{topsurg.fig}
 \end{figure}

It has been known for a while that $- \Sigma(2,3,6n-1)$ supports at most 
$\frac{n(n-1)}{2}$ distinct contact structures up to isotopy. We will denote 
them by $\eta_{i,j}^n$ where $0 \leq i \leq n-2$ and $-n+i+2 \leq j \leq n-i-2$ 
with $j \equiv n-i$  (mod $2$). The geometric meaning of the indices $i$ and 
$j$ will be explained 
in the next section. In order to simplify the exposition we define the 
following notation for the sets of indices of the contact structures $\eta_{i,j}^n$:
\begin{dfn}
For any $n \geq 2$ we define
\[ {\mathcal P}_n = \left \{ (i,j) \in \Z \times \Z : \begin{array}{l} 0 \leq i \leq n-2, \\ 
|j| \leq n-i-2 \text{ with } j \equiv n-i \; ({\rm mod}\, 2) \end{array} \right
\}. \]
\end{dfn}
We can visualize ${\mathcal P}_n$ (and the contact structures indexed by its
elements) as a
triangle with $n-1$ rows and $(n-2,0)$ at its upper vertex. 
For example for $n=5$ we have:

\begin{equation}\label{pascal}
\begin{matrix}
     &      &      &\eta_{3,0\phantom{-}} ^5 &     &     &     \\
     &      & \eta_{2,-1}^5 &     & \eta_{2,1\phantom{-}}^5 &     &     \\
     & \eta_{1,-2}^5 &      & \eta_{1,0\phantom{-}}^5 &     & \eta_{1,2\phantom{-}}^5 &     \\
\eta_{0,-3}^5 &   & \eta_{0,-1}^5 &   & \eta_{0,1\phantom{-}}^5 &   & \eta_{0,3\phantom{-}}^5 \\
\end{matrix}
\end{equation}

For any $n$, the contact structures on the bottom row (i.e. those with $i=0$) 
are obtained by Legendrian surgery on all possible Legendrian realizations of 
the link in Figure \ref{topsurg.fig} (see Figure \ref{legsurg.fig}), and therefore 
are Stein fillable. All other contact structures are strongly symplectically fillable, 
and the top one (i.e. $\eta_{n-2.0}^n$) is known not to be Stein fillable by 
\cite{ghiggini:nonstein}. No Stein filling is known for $\eta^n_{i,j}$
when $i>0$. Therefore we conjecture the following:
 
\begin{conj}
The contact structures $\eta_{i,j}^n$ are not Stein fillable if $i>0$.
\end{conj}

Now we can state the main result of this article:
\begin{theorem}\label{main}
Let $c(\eta^n_{i,j})$ denote the Ozsv\'ath--Szab\'o contact invariant of $\eta^n_{i,j}$.
We can choose representatives for $c(\eta^n_{0,j})$ such that, for any 
$(i,j) \in {\mathcal P}_n$, the contact invariant of $\eta_{i,j}^n$ is computed by 
the formula:
\begin{equation} \label{f:main}
 c(\eta_{i,j}^n)= \sum \limits_{k=0}^i (-1)^k  \binom{i}{k} c(\eta_{0,j-i+2k}^n).
\end{equation}
\end{theorem}

We can reformulate Theorem \ref{main} in plain English as follows.
Any $(i,j) \in {\mathcal P}_n$ determines a sub-triangle ${\mathcal P}_n(i, j) \subset 
{\mathcal P}_n$ with top vertex at $(i,j)$ defined as
\[ {\mathcal P}_n(i, j)= \left \{ (k,l) \in {\mathcal P}_n : 0 \leq k \leq i 
\quad \text{and} \quad j - k \leq l \leq j+k \right \}. \]
The contact invariant of $\eta_{i,j}^n$ is then a linear 
combination of the invariants of the contact structures parametrized
by the pairs in the base of ${\mathcal P}_n(i, j)$. In order to compute the coefficients 
we associate natural numbers to the elements of ${\mathcal P}_n(i, j)$, starting
by associating $1$ to the vertex $(i,j)$, and going downward following the rule
of the Pascal triangle. Then the numbers associated to the elements in the
bottom row, taken with alternating signs, are the coefficients of the contact
invariants of the corresponding contact structures in the sum in Equation 
(\ref{f:main}).

Olga Plamenevskaya proved in \cite{plam:1} that the contact invariants of the
contact structures parametrized by the elements in the bottom row of 
${\mathcal P}_n$ (i.e. those with $i=0$) are linearly independent, so all
$\eta_{i,j}^n$ have distinct contact invariants. Thus we have the following corollary:
\begin{cor}\label{classification}
$- \Sigma (2,3,6n-1)$ admits exactly $\frac{n(n-1)}{2}$ distinct isotopy classes 
of tight contact structures with nonzero  and pairwise distinct Ozsv\'ath--Szab\'o
contact invariants.
\end{cor}
The same classification result could probably be derived also from Wu's work 
on Legendrian surgeries \cite{wu:surgery} and from
the computation of the contact invariants with twisted coefficients of 
contact manifolds with positive Giroux's torsion in \cite{ghiggini-honda:1}. 
However it is not clear how to obtain a complete description of the contact
invariants as in Theorem \ref{main} from that approach.

\subsection*{Acknowledgement}
This work was started when the authors met at the 2008
France-Canada meeting; we therefore  thank the Canadian Mathematical
 Society, the Soci\'et\'e Math\'ematique de France and CIRGET for their hospitality.
 We also thank Ko Honda for suggesting the problem to the first author and helping him 
to work out the upper bound in 2001, and Thomas Mark for useful explanations about
Heegaard Floer homology with twisted coefficients. We finally thank the anonymous referees 
for helping us improve the exposition. The
 first author acknowledges partial support from the ANR project `Floer
 Power.'  
\section{Contact structures on $- \Sigma(2,3,6n-1)$}
\subsection{Construction of the tight contact structures} 
\label{sec:construction}
We introduce the notation 
$$Y_n= - \Sigma (2,3,6n-1)$$ 
and, coherently with the
standard surgery convention, we define $Y_{\infty}$ to be the $3$-manifold obtained 
by $0$-surgery on the right-handed trefoil knot. We describe $Y_{\infty}$ as a quotient of 
$T^2 \times \R$ (with coordinates $(x,y)$ on $T^2$ and $t$ on $\R$):
\[ Y_{\infty} = T^2 \times \R / (\mathbf{v}, t) = (A \mathbf{v}, t-1) \]
where $A \colon T^2 \to T^2$ is induced by the matrix
$\left ( \begin{matrix} 1 & 1 \\
                       -1 & 0
\end{matrix} \right )$.  
% The contact structures 
% $\eta_{i,j}$ will be constructed by Legendrian surgery on contact structures 
% $\xi_i$ on $Y_{\infty}$. 
In \cite{giroux:3} Giroux constructed a family of weakly symplectically
fillable contact structures $\xi_i$ on $Y_{\infty}$ for $i \geq 0$ as follows.
For any $i \geq 0$, fix a function 
$\varphi_i \colon \R \to \R$ such that:
\begin{itemize}
\item[(1)]  $\varphi_i'(t)>0$ for any $t \in \R$, and
\item[(2)] $(2i+1) \pi \leq \sup \limits_{t \in \R} (\varphi_i(t+1)- \varphi_i(t))< 2(i+1) \pi$.
\end{itemize}
By condition $(1)$ the $1$-form 
\[ \alpha_i = \sin(\varphi_i(t))dx+ \cos(\varphi_i(t))dy \]
 defines a contact structure $\tilde{\xi}_i = \ker \alpha_i$ on $T^2 \times \R$. Moreover 
it is possible to choose $\varphi_i$ such that the contact structure $\tilde{\xi}_i$
(but not the $1$-form $\alpha_i$ in general) is invariant under the action 
$({\mathbf v},t) \mapsto (A {\mathbf v}, t-1)$ and therefore defines a contact 
structure $\xi_i$ on $Y_{\infty}$.

\begin{prop}[{\cite[Proposition 2]{giroux:3}}] \label{prop:contact forms}
For any fixed integer $i \ge 0$ the contact structure $\xi_i$ is tight, and its isotopy 
class does not depend on the chosen function $\varphi_i$.
\end{prop}

The knot 
\[ F = \{ \mathbf{0} \} \times \R / (\mathbf{0}, t)=(\mathbf{0}, t-1) \subset Y_{\infty}\]
is Legendrian with respect to $\xi_i$ for any $i$. 
Given a framing on $F$, we define the {\em twisting number} of $F$ with respect to $\xi_i$,
denoted by $tn(F, \xi_i)$, as the number of times $\xi_i|_F$ rotates with respect to the 
framing on $F$. The twisting number depends on the framing and is a generalization of the
Thurston-Bennequin number to knots which are not necessarily null-homologous.

In \cite{ghiggini:3} the first 
author proved the following properties of $F$:
\begin{prop}[{\cite[Lemma 3.5]{ghiggini:3}}]\label{framing on F}
There exists a framing on $F$ such that:
\begin{enumerate} 
\item $tn(F, \xi_i)= -i-1$ 
\item performing surgery on $Y_{\infty}$ along $F$ with surgery coefficient $-n$ 
yields $Y_n$.
\end{enumerate}
\end{prop}
If $Y_\infty$ is identified with the $0$-surgery on the right-handed trefoil knot, then $F$ 
corresponds to a meridian, i.e. it is the knot labeled by ``$-n$'' in Figure \ref{topsurg.fig}.
The framing on $F$ from Proposition \ref{framing on F} then corresponds to the Seifert 
framing of the meridian in the surgery diagram shown in Figure \ref{topsurg.fig}.

Moreover, even though $F$ is nontrivial in homology, we can define a rotation number 
${\tt rot}(\leg, \xi_i)$ for an oriented Legendrian knot $\leg \subset (Y_{\infty}, \xi_i)$ 
smoothly isotopic to $F$: we set ${\tt rot}(F, \xi_i)=0$ for all $i$ and define 
${\tt rot}(\leg, \xi_i) = {\tt rot}(\leg \cup \overline{F}, \xi_i)$, where $\overline{F}$ 
denotes $F$ with the opposite orientation. We do not need to reference a Seifert 
surface for $\leg \cup \overline{F}$ because $c_1(\xi_i)=0$.
%Aside:  this agrees with the rotation number of F, pulled back to S^3.
We are finally in position to give a precise definition of the contact
structures $\eta_{i,j}^n$ and, at the same time, to explain the topological meaning 
of the indices $i$ and $j$.
\begin{dfn}
For any $(i,j) \in {\mathcal P}_n$ the contact manifold $(Y_n, \eta_{i,j}^n)$ is obtained
 by Legendrian surgery on $(Y_{\infty}, \xi_i)$
along a Legendrian knot $F_{i,j}$ which is obtained by applying $n-i-2$ 
stabilizations to $F$, choosing their signs so that ${\tt rot}(F_{i,j}, \xi_i)=j$.
\end{dfn}

In order to complete the classification of tight contact structures on $Y_n$ we
need two steps: 
\begin{enumerate}
\item prove that there are at most $\frac{n(n-1)}{2}$ distinct tight contact 
structures on $Y_n$ up to isotopy, and
\item prove that the contact structures $\eta_{i,j}^n$ are all pairwise nonisotopic.
\end{enumerate}
The first step is a folklore result which follows from the arguments of 
\cite{ghiggini-schonenberger}, but nevertheless we are going to sketch its proof in the 
next subsection. The second step is a corollary of Theorem \ref{main}, 
which will be proved in the last section. 

\subsection{Upper bound}
The upper bound on the number of tight contact structures 
on $Y_n$ can be easily obtained following the strategy in 
\cite{ghiggini-schonenberger}, where the tight 
contact structures on $- \Sigma(2,3,11)$ have been classified. In fact, the manifold
denoted by $Y_n$ in this paper corresponds to the manifold denoted by 
$M(- \frac 12, \frac 13, \frac{n}{6n-1})$ in \cite{ghiggini-schonenberger}.
We recall the conventions of that paper.

\begin{figure}\centering
\psfrag{e0}{\small $0$}
\psfrag{r1}{\small $2$}
\psfrag{r2}{\small $-3$}
\psfrag{r3}{\small $- \frac{6n-1}{n}$}
\includegraphics[width=5cm]{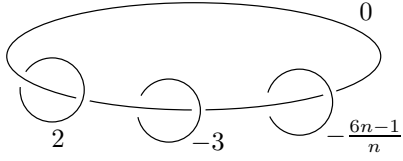}
\caption{Surgery diagram for $- \Sigma(2,3,6n-1)$}
\label{topsurg2.fig}
\end{figure}
The manifold $Y_n$ can be described also by the surgery diagram shown in Figure 
\ref{topsurg2.fig}.
See \cite[Figure 17]{ghiggini-schonenberger} for a sequence of Kirby move
from the diagram in Figure \ref{topsurg2.fig} to the diagram in Figure
\ref{topsurg.fig}.

The surgery diagram \ref{topsurg2.fig} describes a splitting of $Y_n$ into four
pieces:
\[ Y_n = (\Sigma \times S^1) \cup V_1 \cup V_2 \cup V_3 \]
where $\Sigma$ is a three-punctured sphere, i.e. a pair of pants, and $V_1$, $V_2$, 
and $V_3$ are solid tori. We orient the boundary of $\Sigma \times S^1$ by the ``inward 
normal vector first'' convention (i.e. we give it the opposite of the usual boundary 
orientation), and identify each component $\partial(S^1 \times \Sigma)_i$ of 
$\partial(S^1 \times \Sigma)$ with
$\R^2/\Z^2$ by setting $\binom{0}{1}$ as the direction of the $S^1$-fibers and 
$\binom{1}{0}$ as the direction of $\partial(\{pt\}\times \Sigma)$. 

We also fix 
identifications of $\partial V_i$ with $\R^2/\Z^2$ by setting $\binom{1}{0}$ as the 
direction of the meridian. Then we obtain the manifold $Y_n$
by attaching the solid tori $V_i$ to $S^1\times \Sigma$, where the attaching maps
$A_i \colon  \partial V_i \to \partial(S^1\times\Sigma)_i$ are given by
\[ A_1= \left( \begin{matrix} 2 & -1 \\  1 &  0  \end{matrix} \right),\qquad
A_2= \left( \begin{matrix} 3 &  1 \\ -1 &  0  \end{matrix} \right),\qquad
A_3= \left( \begin{matrix}6n-1 &  6 \\ -n & -1  \end{matrix} \right). \]
This construction induces a Seifert fibration on $Y_n$ where the curves
$S^1 \times \{ pt \}$ are regular fibers, and the cores of the solid tori $V_i$ are
the singular fibers. The regular fibers have a natural framing coming from 
the Seifert fibration, and the singular fibers have a framing coming from
the chosen identification of $\partial V_i$ with $\R^2/ \Z^2$. 
These framings can be extended in a unique way to all curves which are 
isotopic to fibers because   the manifolds $Y_n$ are integer homology spheres.
Therefore, for a contact structure $\xi$ on $Y_n$, we can speak about 
the twisting number $tn(L, \xi)$ of a Legendrian curve $L$ which is isotopic 
to a fiber of the Seifert fibration.

% We start by recalling the definition of maximal twisting number of a contact  structure on a Seifert fibered manifold. Let $L$ be a regular fiber for a Seifert fibration on a manifold $M$, and let ${\mathcal S}$ be the set of isotopies $\varphi \colon [0,1] \times M \to M$ such 
% that $\varphi_0$ is the identity and $\varphi_1(L)$ is a Legendrian curve. $L$ has a 
% distinct framing induced by the Seifert fibration, so we can transport 
% this framing to $\varphi_1(L)$. We denote by $L_{\varphi}$ the framed curve $\varphi_1(L)$ 
% with the framing induced by $\varphi$. As a Legendrian curve, $\varphi_1(L)$ has also
% a framing induced by the contact structure. We define the {\em twisting 
% number} $tb(L_{\varphi})$ as the difference between the contact framing and 
% the framing induced by $\varphi$.

 \begin{dfn}
 For any contact structure $\xi$ on $Y_n$, we define 
 the {\em maximal twisting number} of $\xi$ as
 \[
 t(\xi)= \max_{L \in  {\mathfrak L}} \min \{ tn(L, \xi), 0 \}
 \]
where ${\mathfrak L}$ is the set of all Legendrian curves in $Y_n$ which are
smoothly isotopic to a regular fiber.
 \end{dfn}
The maximal twisting number is clearly an isotopy invariant of the contact 
structure $\xi$.

\begin{prop}
Let $\xi$ be a tight contact structure on $Y_n$. Then $t(\xi)<0$.
\end{prop}
\begin{proof}
The proof is the same as in \cite[Theorem 4.14]{ghiggini-schonenberger}.
\end{proof}

\begin{lemma}
If $\xi$ can be isotoped so that the singular fiber $F_2$ is a Legendrian curve 
with twisting number $tb(F_2, \xi)=-1$, then there is a Legendrian regular fiber
with twisting number zero. In particular $\xi$ is overtwisted.
\end{lemma}
\begin{proof}
We isotope $F_1$ so that it becomes a Legendrian curve with twisting number 
$tb(F_1)= k_1 \ll 0$. Let $V_1$ and $V_2$ be standard neighborhoods of $F_1$ and $F_2$
respectively. We assume that $\partial (Y_n \setminus V_1)$ and 
$\partial (Y_n \setminus V_2)$ have Legendrian rulings of infinite slope, and take a convex 
annulus $A$ with boundary on a Legendrian ruling curve of 
$\partial (Y_n \setminus V_1)$ and one of $\partial (Y_n \setminus V_2)$.

The slope of $\partial (Y_n \setminus V_1)$ is $\frac{k_1}{2k_1-1}$, and the slope of 
$\partial (Y_n \setminus V_2)$ is $- \frac 12$. As long as $k_1 \le -1$ the 
Imbalance Principle \cite[Proposition 3.17]{honda:1} provides a bypass along 
a Legendrian ruling curve of $\partial (Y_n \setminus V_1)$. Therefore we can 
apply the Twisting Number Lemma \cite[Lemma 4.4]{honda:1} to increase the 
twisting number $k_1$ of a singular fiber by one up to $k_1=0$, which corresponds
to slope $0$ on $\partial (Y_n \setminus V_1)$. At this point there are two possibilities for
the annulus $A$ between $\partial (Y_n \setminus V_1)$ and $\partial (Y_n \setminus V_2)$: 
either  $A$  carries a bypass for $\partial (Y_n \setminus V_1)$, or it does not. 
If such a bypass exists, then the slope of $\partial (Y_n \setminus V_1)$ can be made 
infinite, and we are done.  
If there is no such a bypass, cutting along $A$ and rounding the edges 
yields a torus with slope $0$ (see \cite[Lemma 3.11]{honda:1}), which is 
$-n$ when measured in $\partial V_3$.
In this case by \cite[Proposition 4.16]{honda:1} we find a convex torus  in $V_3$ with 
slope $- n+ \frac{1}{6}$ ,
which corresponds to infinite slope in $\partial (M \setminus V_3)$.  
\end{proof}

\begin{prop}
Let $\xi$ be a tight contact structure on $Y_n$. Then $t(\xi) = -6k+1$ for some $k$
with $0<k<n-1$.
\end{prop}
\begin{proof}
Let $t(\xi)=-q$.
We start by assuming that the contact structure has been isotoped so that 
there is a Legendrian regular fiber $L$ with twisting number $tb(L, \xi)=-q$, 
and the singular fibers $F_i$ are Legendrian curves with twisting numbers 
$k_i \ll 0$. We take $V_i$ to be a standard neighborhood of the singular fiber 
$F_i$ disjoint from $L$ for $i=1,2,3$. 

The slopes of $\partial V_i$ are $\frac{1}{k_i}$, while the slopes of 
$\partial (Y_n \setminus V_i)$ are $\frac{k_1}{2k_1-1}$, $- \frac{k_2}{3k_2+1}$, and 
$- \frac{nk_3+1}{(6n-1)k_3+6}$ respectively. We also assume that the Legendrian 
ruling on each $\partial (Y_n \setminus V_i)$ has infinite slope, and take 
convex annuli $A_i$ whose boundary consists of $L$ and of a 
Legendrian ruling curve of $\partial (Y_n \setminus V_i)$ for $i=1,2$. If $2k_1-1 < -q$ the 
Imbalance Principle \cite[Proposition 3.17]{honda:1} provides a bypass along 
a Legendrian ruling curve either in $\partial (Y_n \setminus V_1)$ or in 
$\partial (Y_n  \setminus V_2)$. Then we can 
apply the Twisting Number Lemma \cite[Lemma 4.4]{honda:1} to increase the 
twisting number $k_i$ of a singular fiber by one until either $2k_1-1 < -q$, or
$k_1=0$. Similarly we use the annulus $A_2$ to increase $k_2$ until
either $3k_2+1=-q$, or $k_2=-1$.

If $2k_1-1=3k_2+1=-q$ we can write $k_1= -3k +1$, $k_2 = -2k$, and $q=6k-1$ for some
$k > 0$. Take a convex annulus $A$ with Legendrian boundary 
consisting of a Legendrian ruling curve of $\partial (Y_n \setminus V_1)$ and of one of 
$\partial (Y_n  \setminus V_2)$. The dividing set of $A$ contains no boundary parallel arc; 
otherwise we could attach a bypass to either $\partial (Y_n \setminus V_1)$ or to 
$\partial (Y_n  \setminus V_2)$, and the vertical Legendrian ruling curves of the 
resulting torus would
contradict the maximality of $-q$. If we cut $Y_n \setminus (V_1 \cup V_2)$ along $A$ and 
round the edges, we obtain a torus with slope $- \frac{k}{6k+1}$ isotopic to 
$\partial (Y_n \setminus V_3)$. 
Its slope corresponds to $-n +k$ on $\partial V_3$. If $k \geq n$ we can 
find a standard neighborhood $V_3'$ of $F_3$ with infinite boundary slope by
\cite[Proposition 4.16]{honda:1}. This boundary slope becomes $- \frac{1}{6}$ 
if measured with respect to $\partial (Y_n \setminus V_3')$, contradicting $q > 6n-1$.
(Remember that we are assuming $n \geq 2$.)
\end{proof}

\begin{prop}
There are at most $\frac{n(n-1)}{2}$ isotopy classes of tight contact 
structures on $Y_n$.
\end{prop}
 
\begin{proof}
If $t(\xi)= -6k+1$ we can find a neighborhood $V_3$ of $F_3$ such that 
$\partial (Y_n \setminus V_3)$ has slope $- \frac{k}{6k+1}$. This slope corresponds to 
$-n +k$ on $\partial V_3$. By the classification of tight contact structures on solid tori
\cite[Theorem 2.3]{honda:1}, there are $n-k$ tight contact structures on $V_3$. 
Since $k$ ranges from $1$ to $n-1$, we have a total count of at most 
$\frac{n(n-1)}{2}$ tight contact structures on $Y_n$.
\end{proof}

\section{Heegaard Floer homology with twisted coefficients}
In the computation of the contact invariants of $\eta_{i,j}$ we will use the Heegaard
Floer homology groups with twisted coefficients. Since this theory is not as well known as 
the usual Heegaard Floer theory, we give a brief review of its properties. A more detailed 
exposition for the interested reader can be found in the original papers 
\cite{O-Sz:2,O-Sz:3} and in \cite{jabuka-mark}.

Let $Y$ be a closed, connected and oriented $3$-manifold. In the following, singular 
cohomology groups will always be taken with integer coefficients, unless a different abelian 
group is explicitly indicated. Given a module $M$
over the group algebra $\Z[H^1(Y)]$ and a $\spc$-structure $\mathfrak{t} \in \spc(Y)$,
in \cite[Section 8]{O-Sz:2} Ozsv\'ath and Szab\'o defined the Heegaard Floer homology 
group with twisted coefficients $\underline{\hf}(Y, \mathfrak{t} ; M)$, 
which has a natural structure of a  $\Z[H^1(Y)]$-module. When we omit 
the $\spc$-structure from the notation, we understand that we take the direct sum 
over all $\spc$-structures of $Y$. Defining $\underline{\hf}(Y, \mathfrak{t} ; M)$ as a 
$\Z[H_2(Y)]$-module would be a somewhat more natural choice and would lead to simpler 
formulas; however we have chosen to follow the exposition in the original papers.

 Two modules $M$ are of particular interest: the free module of rank one
$M= \Z[H^1(Y)]$, and the module $M=\Z$ with the trivial action of $H^1(Y)$. In the first 
case we will denote $\underline{\hf}(Y; M)=\underline{\hf}(Y)$, and in the second case
$\underline{\hf}(Y; M)=\hf(Y)$.
The automorphism $x \mapsto -x$ of $H^1(Y)$ induces an involution of $\Z[H^1(Y)]$ 
which we call {\em conjugation} and denote by $r \mapsto \bar{r}$. If $M$ is a 
module over $\Z[H^1(Y)]$, we define a new module $\overline{M}$ by taking $M$ as an
additive group, and equipping it with the multiplication $r \otimes m \mapsto \bar{r} \cdot 
m$.

To a cobordism $W$ from $Y_0$ to $Y_1$, in \cite{O-Sz:3}  Ozsv\'ath and Szab\'o 
associated morphisms between the Heegaard Floer homology groups with twisted 
coefficients. However there is an extra complication which is absent
in the untwisted case: the groups are usually modules over different rings, and we need 
to define a ``canonical'' way to transport coefficients across a cobordism. Let us define
\[ K(W) = \ker \left (H^2(W, \partial W) \to H^2(W) \right ). \]
Its group algebra $\Z[K(W)]$ has the structure of both a $\Z[H^1(Y_0)]$-module and 
a $\Z[H^1(Y_1)]$-module induced by the connecting homomorphism $\delta \colon 
H^1(\partial W) \to H^2(W, \partial W)$ for the relative long exact sequence of the pair
$(W, \partial W)$. Therefore we can define the $\Z[H^1(Y_1)]$-module $M(W)$ as
\[ M(W)= \overline{M} \otimes_{\Z[H^1(Y_0)]} \Z[K(W)]. \]

\begin{theorem}[{\cite[Theorem 3.8]{O-Sz:3}}] \label{map-twisted}
Any cobordism $W$ from $Y_0$ to $Y_1$ with a $\spc$-structure $\mathfrak{s} \in
\spc(W)$ induces an anti-$\Z[H^1(Y_0)]$-linear map
\[ F_{W, \mathfrak{s}} \colon \underline{\hf}(Y_0, \mathfrak{s}|_{Y_0}; M) \to
 \underline{\hf}(Y_1, \mathfrak{s}|_{Y_1}; M(W)) \]
which is well defined up to sign, right multiplication by invertible elements of $\Z[H^1(Y_1)]$,
and left multiplication by invertible elements of $\Z[H^1(Y_0)]$.
\end{theorem} \noindent
We will denote the equivalence class of such a map by $[F_{W, \mathfrak{s}}]$. The anti-linearity of
the cobordism maps is a consequence of the unnatural choice of coefficients. The reason for 
it is that $W$ induces the opposite orientation on $Y_0$ and hence a negative sign appears in 
comparing the Poincar\'e duality on $W$ and $Y_0$.

The cobordism maps fit into surgery exact sequences, of which we state only 
the simple case we use in the paper. 
\begin{theorem}[{\cite[Theorem 9.21]{O-Sz:2}}; cf. 
{\cite[Section 9]{jabuka-mark}}] \label{twisted-surgery} 
Let $Y$ be an integer homology sphere and $K \subset Y$ a knot. We identify framings 
on $K$ with integer numbers by assigning $0$ to the framing induced by an embedded 
surface with boundary in $K$, and denote by $Y_n(K)$
the manifold obtained by performing $n$ surgery along $K$.  Then there is an
exact triangle
\[ \xymatrix{
\hf(Y)[t^{-1},t] \ar[rr]^{F} & & \hf(Y_{-1}(K))[t^{-1},t] \ar[dl] \\
& \underline{\hf}(Y_0(K)) \ar[ul]
} \]
If $W$ is the $4$-dimensional cobordism from $Y$ to $Y_{-1}(K)$ induced by the 
surgery, $[\widehat{\Sigma}]$ is a generator of $H_2(W)$, and $\mathfrak{s}_k$ is the 
unique $\spc$-structure on $W$ such that 
$\langle c_1(\mathfrak{s}_k), [\widehat{\Sigma}] \rangle = 2k+1$, then
\[ F= \sum \limits_{k \in \Z} F_{W, \mathfrak{s}_k} \otimes t^k. \]
\end{theorem}
 
Maps between Heegaard  Floer homology groups with twisted coefficients satisfy 
composition formulas which are both more involved and more powerful than the analogous
formulas for ordinary coefficients. The source of the difference is that, given cobordisms 
$W_0$ from $Y_0$ to $Y_1$ and $W_1$ from $Y_1$ to $Y_2$, the coefficient ring $M(W)$
associated to the map $F_W$ induced by the composite cobordism $W=W_1 \cup W_0$ is
usually smaller than the coefficient ring $M(W_0)(W_1)$ associated to the composition 
$F_{W_1} \circ F_{W_0}$. More precisely:
\begin{lemma} \label{Ksequence}
There is an exact sequence
\begin{equation} \label{eqn: Ksequence} 
0 \to K(W) \stackrel{\iota} \to \dfrac{K(W_0) \oplus K(W_1)}{H^1(Y_1)} \to 
{\rm Im}(\delta)  \to 0
\end{equation}
where $\delta \colon H^1(Y_1) \to H^2(W)$ is the connecting homomorphism for the
Mayer--Vietoris sequence of the triple $(W_0, W_1, W)$.
\end{lemma}
\begin{proof}
The exact sequence (\ref{eqn: Ksequence}) follows from the commutative diagram: 
\[ \xymatrix{
 & H^2(W_0) \oplus H^2(W_1) & H^2(W) \ar[l] & H^1(Y_1) \ar[l]_{\delta} \\
H^1(Y_1) \ar[r] & H^2(W_0, \partial W_0) \oplus H^2(W_1, \partial W_1) \ar[r]  \ar[u] & 
H^2(W, \partial W)  \ar[u] & } \]
where the top row is the Mayer--Vietoris sequence and the 
bottom row is the relative cohomology sequence for the pair $(W, Y_1)$. In fact 
$H^2(W_0, \partial W_0) \oplus H^2(W_1, \partial W_1)  \cong H^2(W, Y_0)$ by homotopy 
equivalence and excision.
\end{proof}

 The inclusion $\iota \colon K(W) \to
\dfrac{K(W_0) \oplus K(W_1)}{H^1(Y_1)}$ gives rise to a projection 
\[\Pi \colon \Z[K(W_0)] \otimes_{\Z[H^1(Y_1)]} \Z[K(W_1)] \cong \Z \left [  
\dfrac{K(W_0) \oplus K(W_1)}{H^1(Y_1)} \right ] \longrightarrow \Z[K(W)] \]
defined by 
\[ \Pi(e^w) = \left \{ \begin{array}{l} e^w \quad \text{if } w=\iota(v) \text{ for some } v \\
0 \quad \text{otherwise.}
\end{array} \right. \]
which extends to a projection $\Pi_M \colon M(W_0)(W_1) \to M(W)$ for any
$\Z[H^1(Y_0)]$-module $M$. The composition law for twisted coefficients can be stated 
as follows.
\begin{theorem}[{\cite[Theorem 3.9]{O-Sz:3}}; cf.
{\cite[Theorem 2.9]{jabuka-mark}}] \label{twisted-composition}
Let $W=W_0 \cup_{Y_1} W_1$ be a composite cobordism with a $\spc$-structure 
$\mathfrak{s}$. Write $\mathfrak{s}_i=\mathfrak{s}|_{W_i}$. Then there are choices
of representatives for the maps $F_{W_0}$, $F_{W_1}$, and $F_W$ such that:
\begin{equation} \label{eqn: twisted composition}
[F_{W, \mathfrak{s}+ \delta h}] = [\Pi_M \circ F_{W_2, \mathfrak{s}_2} \circ 
e^{-h} \cdot F_{W_1, \mathfrak{s}_1}]
\end{equation}
where $h \in H^1(Y_1)$ and $\delta \colon H^1(Y_1) \to H^2(W)$ is the connecting 
homomorphism for the Mayer--Vietoris sequence.
\end{theorem}

To a contact structure $\xi$ on $Y$ we can associate an element $c(\xi,M) \in 
\underline{\hf}(-Y, \mathfrak{t}_{\xi})$ where $-Y$ denotes $Y$ with the opposite 
orientation, and $\mathfrak{t}_{\xi}$ denotes the canonical $\spc$-structure on $Y$ 
determined by $\xi$. This contact element is well defined up to sign and multiplication by
invertible elements in $\Z[H^1(Y)]$, and we will denote $[c(\xi,M)]$ its equivalence class.
When $M$ is clear from the context we will drop it from the notation.
\begin{theorem}[Ozsv\'ath--Szab\'o \cite{O-Sz:cont}] Let $\xi$ be a contact structure on 
$Y$. Then:
\begin{enumerate}
\item $[c(\xi,M)]$ is an isotopy invariant of $\xi$,
\item if $c_1(\xi)$ is a torsion cohomology class, then $[c(\xi,M)]$ is a set of
homogeneous elements of degree $- \frac{\theta(\xi)}{4} - \frac 12$, where $\theta$ is
Gompf's $3$-dimensional homotopy invariant defined in \cite[Definition 4.2]{gompf:1},
\item if $\xi$ is overtwisted, then $[c(\xi,M)]=\{ 0 \}$.
\end{enumerate}
\end{theorem}
The behaviour of  the contact invariant is contravariant with respect to Legendrian
surgeries, as described by the following theorem:
\begin{theorem}\label{natural}
Let $(Y_0, \xi_0)$ and $(Y_1, \xi_1)$ be contact manifolds, and let
$(W, J)$ be a Stein cobordism from $(Y_0, \xi_0)$ to $(Y_1, \xi_1)$\footnote{Our convention is
that Stein cobordisms always go from the negative end to the positive end. This is the natural convention 
from the point of view of topology. Some authors use the opposite convention, which is more natural for 
symplectic field theory.}  which is 
obtained by Legendrian surgery on some Legendrian link in $Y_0$. If $\mathfrak{k}$ is the 
canonical $\spc$-structure on $W$ for the complex structure $J$, then:
\[ [F_{W, \mathfrak{s}}(c(\xi_1, M))] = \left \{ \begin{array}{l} [c(\xi_0, M(W))] \text{ if }
\mathfrak{s} = \mathfrak{k} \\ \{ 0 \} \text{ if } \mathfrak{s} \ne \mathfrak{k}.
\end{array} \right. \]
\end{theorem}
This theorem is essentially due to Ozsv\'ath and Szab\'o \cite{O-Sz:cont}, but an explicit 
statement has been given by Lisca and Stipsicz \cite{lisca-stipsicz:4} for untwisted 
coefficients. Here we state a generalization of \cite[Lemma 2.10]{ghiggini:3} to twisted 
coefficients and to links with more than one component. However the proof remains 
unchanged. 

In Theorem \ref{natural} the map $F_{W, \mathfrak{s}}$ is actually induced by the {\em
opposite} cobordism, which goes from $-Y_1$ to $-Y_0$, and which is often denoted by 
$\overline{W}$. We chose to drop this extra decoration from the notation because, in the 
computation of the Ozsv\'ath--Szab\'o invariants, our maps will {\em always} be induced by
 the opposite of the cobordisms constructed by Legendrian surgeries.

For any contact structure $\xi$ on a $3$--manifold $Y$ we denote by 
$\overline{\xi}$ the contact structure on $Y$ obtained from $\xi$ by inverting the 
orientation of the planes.  This operation is called {\em conjugation}.
 In Heegaard Floer homology  there is an involution
$\mathfrak{J}: \widehat{HF}(-Y) \to \widehat{HF}(-Y)$ 
 defined in \cite[Section 2.2]{O-Sz:2} and \cite[Section 5.2]{O-Sz:3} which is 
closely related to conjugation of contact structures. We are going to state and use its 
main property only for the untwisted version of the contact invariant.
\begin{theorem}[{\cite[Theorem 2.10]{ghiggini:3}}] \label{coniugazione}
Let $(Y, \xi)$ be a contact manifold. Then 
\[ c(\overline{\xi})= \mathfrak{J}(c(\xi)). \]
\end{theorem}
We end this section with a remark which explains how, under certain circumstances, it is 
possible to re-interpret the cobordism maps as {\em linear} maps by making the appropriate 
identifications between the coefficient rings.
\begin{lemma} \label{simplifymap}
Let $\iota_0 \colon Y_0 \to W$ and $\iota_1 \colon Y_1 \to W$ be the inclusions. 
If the maps $(\iota_0)_* \colon H_2(Y_0) \to H_2(W)$ and $(\iota_1)_* \colon H_2(Y_1) 
 \to H_2(W)$ are injective and ${\rm Im}(\iota_0)_* = {\rm Im}(\iota_1)_*$ 
we can define an isomorphism
$(\iota_W)_* = (\iota_1)_*^{-1} (\iota_0)_* \colon H_2(Y_0) \to H_2(Y_1)$. After 
composing with Poincar\'e dualities, we obtain an isomorphism
\[ \iota_W^! \colon H^1(Y_0) \to H^1(Y_1), \]
which induces a structure of $\Z[H^1(Y_0)]$-module on $\Z[H^1(Y_1)]$. Moreover with 
this structure there is an anti-$\Z[H^1(Y_0)]$-linear isomorphism  
\[ \overline{\Z[H^1(Y_0)]} \otimes_{\Z[H^1(Y_0)]} \Z[K(W)] \cong \Z[H^1(Y_1)]. \]
\end{lemma}
When the hypothesis of Lemma \ref{simplifymap} is satisfied we can interpret the 
cobordism 
map $F_W$ as a $\Z[H^1(Y_0)]$-{\em linear} map
\[ F_W \colon \underline{\hf}(Y_0) \to \underline{\hf}(Y_1). \]
\begin{proof}
Let us decompose 
$$\delta \colon H^1(\partial W) = H^1(Y_0) \oplus H^1(Y_1) \to K(W)$$
as $\delta = \delta_0 \oplus \delta_1$. Each $\delta_i \colon H^1(Y_i) \to K(W)$ is an
isomorphism because $\delta \colon H^1(\partial W) \to H^2(W, \partial W)$ corresponds 
to $\iota \colon H_2(\partial W) \to H_2(W)$ by Poincar\'e-Lefshetz duality. (See 
\cite[Theorem~28.18]{ga}). 
Then we have a commutative diagram:
 \[ \xymatrix{
H^1(Y_0) \ar[rr]^{- \iota_W^!} \ar[rd]_{\delta_1} &  &  H^1(Y_1) \ar[ld]^{\delta_2} \\
& K(W) & } \]
In fact, let $c_0$ be the Poincar\'e dual of $a \in H^1(Y_0)$, and $c_1$ the Poincar\'e
dual of $\iota^!_W(a) \in H^1(Y_1)$. Then $(\iota_1)_*(c_1)- (\iota_0)_*(c_0) =0$ in
$H_2(W)$, which implies that $c_1 - c_0 = \partial C$ for some class 
$C \in H_3(W, \partial W)$. Taking Poincar\'e duals on $W$ and $\partial W$ we obtain that
$a + \iota_W^!(a)$ is in the image of the restriction map 
$H^1(W) \to H^1(\partial W)$ --- the change of sign because $W$ induces the opposite 
orientation on  $Y_0$. Therefore $\delta(a + \iota_W^!(a))=0$, from which the 
commutativity of the diagram follows.

The isomorphism between $\overline{\Z[H^1(Y_0)]} \otimes_{\Z[H^1(Y_0)]} \Z[K(W)]$ and 
$\Z[H^1(Y_1)]$ is given by the map 
\[ e^a \otimes e^b \mapsto e^{\iota_W^!(a)+ \delta_2^{-1} (b)}. \]
This map is well defined by the universal property of the tensor product, because it is 
induced by a $\Z[H^1(Y_0)]$-bilinear map  
\[ \overline{\Z[H^1(Y_0)]} \times  \Z[K(W)] \to \Z[H^1(Y_0)]. \]
In fact for all $a' \in H^1(Y_0)$ we have 
\begin{align*} 
(e^{a-a'}, e^b) & \mapsto e^{\iota_W^!(a-a')+ \delta_2^{-1}(b)} \\
(e^a, e^{\delta_1(a')+b}) & \mapsto e^{\iota_W^!(a) + \delta_2^{-1}(\delta_1(a')+b)},
\end{align*}
but $e^{\iota_W^!(a-a')+ \delta_2^{-1}(b)}= e^{-a'} \cdot e^{\iota_W^!(a)+ \delta_2^{-1}(b)}$.
\end{proof}
\section{Computation of the Ozsv\'ath--Szab\'o contact 
invariants}
We are going to sketch the strategy of the computation as a guide for the 
reader. The topological input is a Legendrian surgery along a Legendrian link
$\leg \cup \mathcal{C}$ which takes the contact manifold $(Y_{\infty}, \xi_{i+1})$ 
to $(Y_n, \eta_{i,j}^n)$.  This Legendrian surgery factors in two ways, one through 
$(Y_{\infty}, \xi_i)$ and one through $(Y_{n+1}, \eta_{i+1,j}^{n+1})$, depending on 
whether we perform the surgery first along $\leg$, and then along $\mathcal{C}$, or 
vice versa.  The knot $\leg$ is a stabilization of $F$ and $\mathcal{C}$ is a link which 
is naturally Legendrian in each $(Y_\infty, \xi_i)$ for $i>0$. 

Then we have homomorphisms in Heegaard Floer homology mapping the 
invariants of the tight contact structures on $Y_n$ to the invariants of the 
tight contact structures on $Y_{n+1}$ above the bottom row of the triangle 
${\mathcal P}_{n+1}$. We compute these invariants by an inductive 
argument using the fact that the invariants of the tight contact structures
on the bottom row span $\widehat{HF}(- Y_n)$ in the appropriate degree 
\cite{plam:1}.

A feature of the computation is that it requires the use of Heegaard Floer
homology with twisted coefficients.
% namely the contact invariants of $\xi_i$ with twisted coefficients, which were computed 
% in \cite{ghiggini-honda:1}. 
This is somewhat surprising as the manifolds $Y_n$ are integer homology
spheres and therefore they carry no nontrivial twisted coefficient system. The reason of the
effectiveness of twisted coefficients is twofold. On the one hand the large indeterminacy 
of the contact invariant with twisted coefficients allows the contact invariants 
$c(\eta^n_{0,j})$ to be mapped to different representatives of $c(\xi_0)$ --- see Lemma 
\ref{pluto}. On the other hand the contact invariants  of $(Y_{\infty}, \xi_i)$ with twisted 
coefficients are all nonzero and pairwise distinct for $i \ge 0$, while the untwisted ones 
vanish for $i>0$ by \cite[Theorem 1]{ghiggini-honda-vhm}. 

%%%%%%%%%%%%%%%%%%%%%%%%%%%%%%%%%%%%%%%%%%
\subsection{The surgery construction}\label{construction}
%%%%%%%%%%%%%%%%%%%%%%%%%%%%%%%%%%%%%%%%%%

We find the Legendrian link $\leg \cup \mathcal{C}$ by studying open book 
decompositions of $(Y_\infty, \xi_i)$ and $(Y_n, \eta_{i, j}^n)$.  All knots will be oriented 
and $\overline{K}$ will be used to denote $K$ with its orientation reversed.

\begin{figure}[htp]
	\labellist
	\small\hair 2pt

%primary labels
	\pinlabel $F$ at 52 183
%	\pinlabel {$\left\{$} at 103 301
	\pinlabel $\leg_{l,r}$ at 27 179
     \pinlabel {$r$} at 86 281
	\pinlabel {$l$} at 163 309

	\pinlabel {$\left\{ \phantom{\begin{matrix} B \\B \end{matrix} }\right.$} at 102 281
	\pinlabel {$\left\{ \phantom{\begin{matrix} B \\B \end{matrix} }\right.$} at 180 309
	\pinlabel \rotatebox{90}{$ \left\{ \phantom{\begin{matrix} 
B \\B\\B\\B\\B\\B\\B\\B\\B\\B\\B\\B\\B\\B \end{matrix} }\right.$} at 208 -5
	\pinlabel {$i$ times} at 208 -30 

%negative meridional twists
	\pinlabel - at 81 159
	\pinlabel - at 137 159
	\pinlabel - at 149 159
	\pinlabel - at 202 159
	\pinlabel - at 214 159
	\pinlabel - at 270 159
	\pinlabel - at 282 159
	\pinlabel - at 336 159
	\pinlabel - at 82 208
	\pinlabel - at 153 208
	\pinlabel - at 229 208
	\pinlabel - at 284 208
	\pinlabel - at 341 208
%positive standard punctures	
	\pinlabel + at 123 62
	\pinlabel + at 192 62
	\pinlabel + at 259 62
	\pinlabel + at 324 62
	\pinlabel + by -.25 0 at 127 239 % stabilized
	\pinlabel + by -.25 0 at 200 245 % stabilized
	\pinlabel + at 268 270
	\pinlabel + at 323 271
	\pinlabel + at 377 270 
%stabilizations
	\pinlabel + at 123 289
	\pinlabel + at 123 271	
	\pinlabel + at 197 317
	\pinlabel + at 197 299
	\endlabellist
\centering
\includegraphics[height=3in]{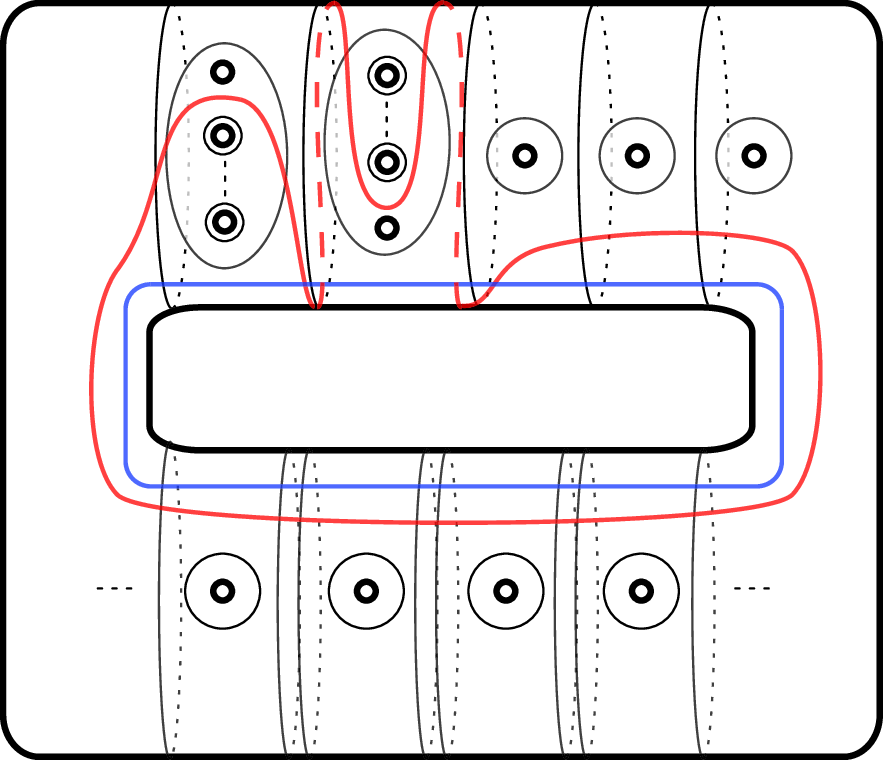}
\vspace{20pt}
\caption{\label{fig:books} Genus one open book decompositions of the contact 
manifolds $(Y_\infty, \xi_i)$ and $(Y_n,\eta_{i,j}^n)$ where $j = l-r$ and $n = l+r+i+2$.  
The thick circles are the boundary of the page and the label $i$ determines how many 
times the region along the bottom should be repeated.  To get an open book compatible 
with $(Y_\infty, \xi_i)$, take the monodromy as the product of Dehn twists along all curves, 
except $F$ and $\leg_{l,r}$, with signs as indicated.  Add a positive twist along the curve 
$\leg_{l,r}$ to get $(Y_n, \eta_{i,j})$.  The curve $F$ is shown for comparison.} 
\end{figure}
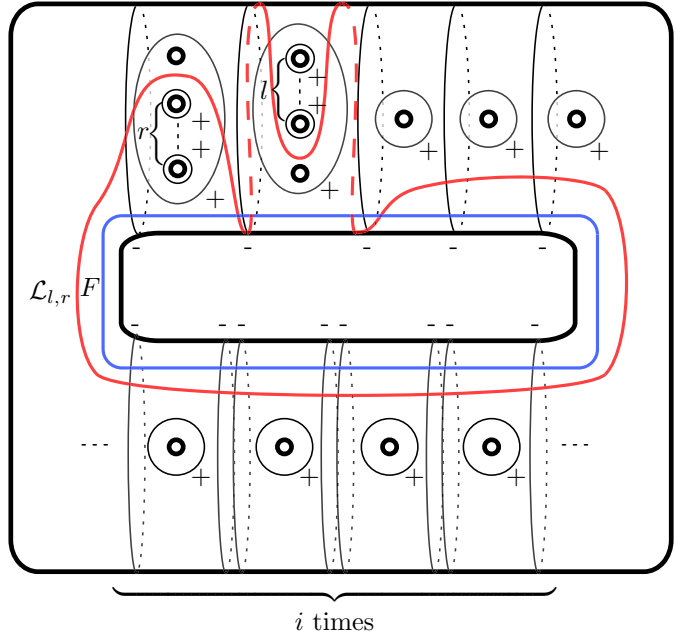

\begin{prop} \label{prop:surgery}  There is a Legendrian link $\leg_{l,r} \cup 
\mathcal{C}$ in $(Y_\infty, \xi_{i+1})$, for every $i \ge 0$, so that Legendrian surgery 
along $\leg_{l,r} \cup \mathcal{C}$ gives the contact manifold $(Y_n, \eta_{i,j}^n)$, 
while surgery along $\leg_{l,r}$ gives the contact manifold $(Y_{n+1}, 
\eta_{i+1, j}^{n+1})$, where $j = l-r$ and $n = l+r+i+2$.  For every $i$, the links 
$\leg_{l,r} \cup \mathcal{C}$ are smoothly isotopic in $Y_\infty$.  Further, the image of 
$\leg_{l,r}$ in $\xi_{i+1}$ under the surgery along $\mathcal{C}$ can be identified with 
$\leg_{l,r}$ in $\xi_i$.
\end{prop}

We will prove Proposition \ref{prop:surgery} by first constructing open book 
decompositions compatible with $(Y_\infty, \xi_{i+1})$ which have the Legendrian knot 
$F$ sitting naturally on a page, and see how to stabilize $F$ to get the knot $\leg_{l,r}$, 
still sitting on the page of a compatible open book.  We then show how to modify this 
open book by adding positive Dehn twists to its monodromy to get an open book 
compatible with $\xi_{i}$, noting that this takes the knot $F$ in $\xi_{i+1}$ to the knot $F$ in 
$\xi_{i}$.  

We deal primarily with open books in their abstract form: as a surface $S$ with boundary 
together with a self-diffeomorphism $\phi$, usually presented as a product of Dehn twists
along curves labeled either $+$ or $-$ on a diagram of $S$.  In Figure~\ref{fig:books}, such a 
diagram is given for the contact manifold $(Y_\infty, \xi_i)$.  The surface $S$ is a torus 
with $3i+n+3$ open disks removed. The monodromy consists of positive 
(right-handed) Dehn twists about circles parallel to (most) boundary circles of $S$ 
together with negative (left-handed) Dehn twists about certain pairwise disjoint curves 
which intersect $F$ once. We will call these latter curves 
{\em meridians.}    
In \cite{vanhorn} (see Section 4.4),
it was shown how such open books corresponded to torus bundles as well as how these 
open books could be embedded. We discuss some of that procedure here.  

The total space $Y_\infty$ is a torus bundle. We can see the bundle structure directly from the open book. Begin by looking at a meridian $c$ on the torus $S$ in Figure~\ref{fig:books} which is disjoint from those curves used in the presentation of the monodromy. 
Since $c$ is fixed by the monodromy, it traces out a torus which will be a fiber in the
torus bundle $Y_\infty$.  As we move $c$ around the torus 
page it traces out a family of torus fibers. When $c$ crosses the meridional (negative) Dehn twists, this induces a (negative) Dehn twist along the torus fiber, along a curve 
parallel to the page of the open book. Crossing a boundary circle with a positive Dehn 
twist induces a 
negative Dehn twist along the fiber, this time along a direction orthogonal to that of the 
page (see \cite[Section 4.2]{vanhorn}).  These two Dehn twists correspond to the 
standard Dehn twist generators of the mapping class group of the torus and allow one 
to construct all the universally tight, linearly twisting contact structures on torus bundles 
(i.e. those described in Proposition \ref{prop:contact forms}).  

The region above the bracket labeled `$i$ times' shows an 
open book compatible with a region of Giroux torsion one (multiplied $i$ times).  When 
$i=0$, the open book describes the unique Stein fillable contact structure 
(see \cite[Section 4.5]{vanhorn}).  Each of the pieces of the open book swept out as $c$ 
passes a boundary component is a bypass and is compatible with a 
linear contact form of type used in Proposition \ref{prop:contact forms}.  The 
horizontal arcs that make up the curve $F$ in Figure \ref{fig:books} are linear in these 
local models (see \cite[Figure 4.5] {vanhorn}) and correspond to a Legendrian $\{ pt\}
 \times I$ in $T^2 \times I$.  Section 4.5 of \cite{vanhorn} constructs our particular open 
books and shows they are compatible with the given contact structures; the manifold 
$Y_\infty$ is the torus bundle with monodromy $T^{-1}S = \left( \begin{array}{cc} 1 & 1 
\\ -1 & 0 \end{array} \right)$.  (This is different than what is stated in 
\cite[Section 4.7.2]{vanhorn}.  The second author gave there a description 
of the torus bundle with the wrong orientation (0 surgery on the left-handed trefoil).)

We will prove this compatibility later in Proposition \ref{prop:equivalence}.
The first thing we need for Proposition \ref{prop:surgery} is a way to stabilize $F$ on 
the page of a compatible open book.

\begin{figure}
	\labellist	
	\small\hair 2pt
	\pinlabel $\leg$ at 220 80
	\pinlabel $\leg_-$ at 40 155
	\pinlabel $\leg_+$ at 40 25
	\endlabellist
\centering
\includegraphics{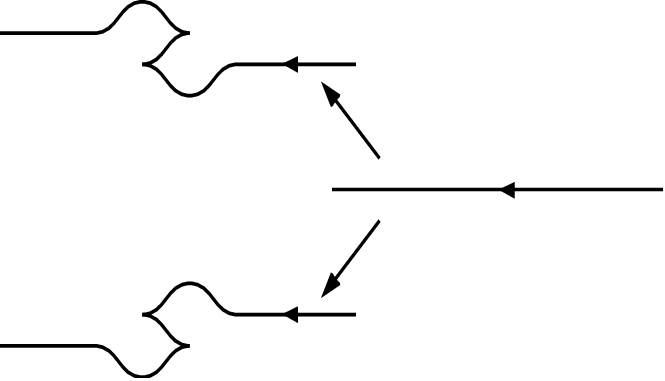}
\caption{ \label{fig:leg stab} Local picture of the front projection for two Legendrian 
stabilizations.}
\end{figure}

\begin{dfn}  Let $\leg$ be a Legendrian knot.  There are two stabilizations of $\leg$, 
positive and negative, denoted by $\leg_+$ and $\leg_-$, resp., given by the front 
projections shown in Figure \ref{fig:leg stab}.
\end{dfn}

Oriented to the right as shown in Figure \ref{fig:leg stab}, $\leg_-$ and $\leg_+$ have 
$${\tt rot}(\leg_- \cup \overline{\leg}) =  -1$$ and
$${\tt rot}(\leg_+ \cup \overline{\leg}) = +1.$$ 

\begin{dfn}  Let $L$ be a knot on a page of an open book.  There are two stabilizations, 
left and right, denoted by $L_l$ and $L_r$, resp., given by the local pictures shown in 
Figure \ref{fig:book stab}.  
\end{dfn}

\begin{figure}[htp]
\labellist
\pinlabel $P$ at 27 39
%\pinlabel $\Sigma$ at 140 52
\pinlabel $\knot_r$ at 88 55
\pinlabel $\knot_l$ at 157 23
\pinlabel $\knot$ at 55 45
\pinlabel $+$ at 73 51
\pinlabel $+$ at 189 29 
\endlabellist
\centering
\includegraphics[width=4.in]{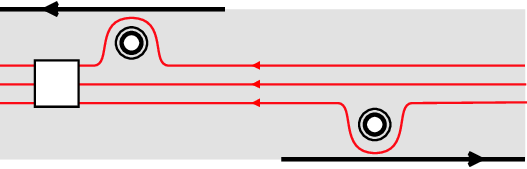}
\caption{ \label{fig:book stab} Local picture of the two stabilizations of a knot on a page 
of an open book.  All knots $\knot$, $\knot_l$ and $\knot_r$ are oriented to the left, so 
that $\knot_l$ is obtained by sliding $\knot$ across a trivial stabilization on its left, and 
similarly $\knot_r$ on its right.  The patch $P$ is used in Lemma \ref{lm:zig zag} to 
prove that open book stabilizations give rise to Legendrian stabilizations.  The embedded 
knots $\knot$, $\knot_l$ and $\knot_r$ are smoothly isotopic.}
\end{figure}
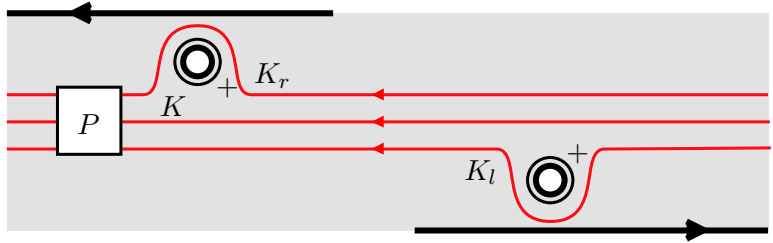

We need the following lemma and will sketch its proof.  A more complete proof can be 
found in \cite{Onaran}.

\begin{lemma}\label{lm:zig zag}  Let $\knot$ be a nonisolating knot on a page of an 
open book.  If $\leg$ is the Legendrian realization of $\knot$, then the Legendrian 
realization of $\knot_r$ is the negative stabilization $\leg_-$ and the Legendrian realization 
of $\knot_l$ is the positive stabilization $\leg_+$.  
\end{lemma}  

\begin{proof}[Sketch of proof]  First, observe that we can make $\knot$, $\knot_l$ and 
$\knot_r$ simultaneously Legendrian while sitting on the same page $\Sigma$ of the open 
book, and we let $\leg$, $\leg_l$ and $\leg_r$ refer to these Legendrian knots.  All three 
knots are smoothly isotopic and so $\knot \cup \overline{\knot}_l$ bounds an annulus, 
$A$.  We can make this annulus convex with Legendrian boundary $\leg \cup 
\overline{\leg}_l$, starting with the patch $P$ as shown
 in Figure \ref{fig:book stab}, a subset of the page.  Since the dividing set is empty on $\Sigma$, it is empty on 
$P$ and so the dividing set of $A$ consists of 
boundary parallel arcs adjacent to $\knot_l$. Comparing framings shows that there is only 
one such arc and so ${\tt rot}(\knot \cup \overline{\knot}_l) = \xi(A_+) - \xi(A_-) = -1 = -{\tt rot}(\knot_l \cup \overline{\knot})$.  
This shows that $\knot_l$ is the positive stabilization.
\end{proof}

We will also need the following tool regarding stabilizations of open books.  

\begin{lemma}[The braid relation] \label{lm:braid} Two open books which locally differ as 
in Figure \ref{fig:braid} are related by a positive Hopf stabilization.  There is a contact structure $\xi$ defined in a neighborhood of the local picture compatible with both open books and such that the horizontal arc $K$ is Legendrian and sitting on a page in each.
\end{lemma}

 \begin{figure} 
 \begin{center}

 		\mbox{\includegraphics[width=5cm]{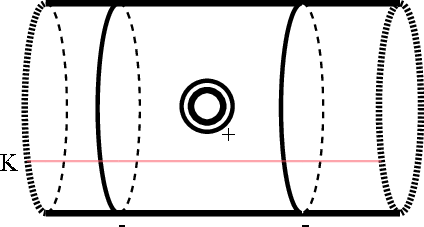}}
 		\hspace{.5in}
 		\mbox{\includegraphics[width=5cm]{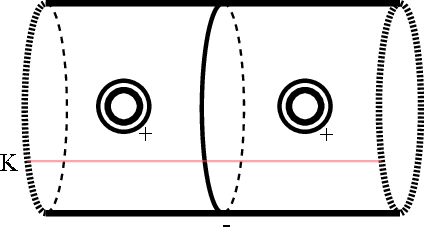}}
		
        \caption{Local pictures $B_1$ and $B_2$ of open books which differ by a Hopf 
        stabilization.}
 	\label{fig:braid}
 \end{center}
 \end{figure}
 
 \begin{figure}[htp]
	\begin{center}
	
	\labellist
	\small\hair 2pt
	\pinlabel $\sigma_1$ at 131 95
	\pinlabel $\sigma_2$ at 38 100
	\pinlabel $\sigma_3$ at 93 7
	\endlabellist
	\includegraphics[width=5cm]{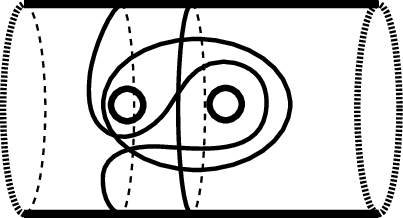}
	\hspace{.5 in}
	\labellist
	\pinlabel $\partial_1$ at 116 42
	\pinlabel $\partial_2$ at 64 42
	\pinlabel $\partial_3$ at 27 5
	\pinlabel $\partial_4$ at 166 5
	\endlabellist
	\includegraphics[width=5cm]{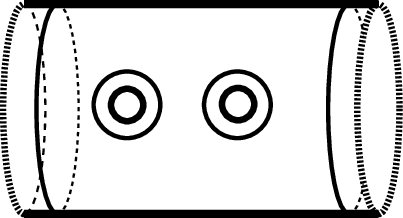}
	\caption{The lantern relation}
	\label{fig:lantern relation}
	\end{center}
 \end{figure}

 \begin{figure} 
 \begin{center}
 		\mbox{\includegraphics[width=5cm]{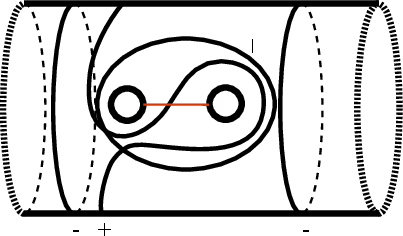}}
		
 	\caption{The monodromy of $B_2$ after applying the lantern relation. 
 The obvious destabilizing arc is shown. }

 	\label{fig:lantb}
 \end{center}
 \end{figure}

\begin{proof} The lantern relation (shown in Figure \ref{fig:lantern relation}) relates the 
product of right-handed Dehn twists along each boundary component to the product of 
those along the three interior curves: $\partial_1 \partial_2 \partial_3 \partial_4 = \sigma_1 
\sigma_2 \sigma_3$, (where the Dehn twists act left to right).  This diagram is different
than the usual presentation of the lantern relation which draws the surface as a 
three-holed disk with the curves $\sigma_i$ placed symmetrically, cf. 
\cite{dehn38, johns79}, but is more convenient for our purposes here.

The segment $B_2$ of the open book shown on 
the right hand side of Figure \ref{fig:braid} is a 4-holed sphere with monodromy
$\partial_1 \partial_2 \sigma_3^{-1}$ using the same curves as in Figure \ref{fig:lantern 
relation}.  After applying the lantern relation to $B_2$ we get the presentation 
$\sigma_1 \sigma_2 \partial_3^{-1} \partial_4^{-1}$.
In applying the lantern relation here it is important that all $\partial_i$ commute with 
each other and with all $\sigma_i$.
The new presentation is shown in Figure 
\ref{fig:lantb} with an obvious destabilizing arc.  After destabilizing, we are left with the 
open book segment $B_1$.  Notice that the destabilizing arc is disjoint from the  
horizontal arc $K$ shown in Figure \ref{fig:braid} and so we can apply the braid relation 
even when there are Dehn twists along curves running parallel to the segment, so long as 
the Dehn twists along $\partial_1$ and $\partial_2$ occur simultaneously in the described 
factorization.  It was shown in \cite[Section 4.2]{vanhorn} how to construct a contact 
form compatible with $B_1$.  Gluing two of these together gives a contact form compatible 
with both $B_2$ and $B_1$ and with the horizontal arc $K$ being Legendrian and sitting 
on pages of each.
\end{proof}

\begin{prop}\label{prop:equivalence}  The open books described in Figure \ref{fig:books} 
are compatible with the contact structures $\xi_i$ on $Y_\infty$ with the Legendrian knot 
$F$ from Proposition \ref{prop:contact forms} sitting on the page as the knot $F$ in the 
figure.
\end{prop}

\begin{proof}  This can and is proved without resorting to many of the intricacies 
discussed in the beginning of the section.  We first note that the region labeled $i$ 
times is a region of Giroux torsion 1, multiplied $i$ times, and for convenience, let 
us denote the associated open book $\mathfrak{B}_i$.  From 
\cite[Lemma 4.4.4]{vanhorn} and its corollary, we see that the compatible contact 
structures  are weakly 
fillable for all $i$ and hence by the classification in \cite{giroux, honda:2} must be in the 
Giroux's family of tight contact structures constructed in Section \ref{sec:construction}.  
From \cite[Section 4.5]{vanhorn} --- recalling that 
the monodromy for the right-handed trefoil is $T^{-1}S = \left( \begin{array}{cc} 1 
& 1 \\ -1 & 0 \end{array} \right)$ --- we see that, when $i=0$, the compatible contact 
structure has zero Giroux torsion. (Indeed, it can be realized by Legendrian surgery 
on the Stein fillable contact structure on $T^3$).  Thus $\mathfrak{B}_i$ is compatible 
with $\xi_i$.  To see that the curve $F$ in the diagram really is the Legendrian $F$ 
discussed after Proposition \ref{prop:contact forms}, we do need a bit of detail.  
Looking at \cite[Figure 4.5]{vanhorn}, the embedded diagram of a basic slice is 
compatible with a linearly twisting contact form of the type discussed in Proposition 
\ref{prop:contact forms}, and further the arc tangent to the $t$-axis at the left and 
right sides of the picture is Legendrian.  One can glue any number of basic slices 
together (as well as gluing the front and back boundaries together) and the resulting 
open book will still be compatible with a linear contact form and the matched horizontal 
arc will remain Legendrian. 
\end{proof}

\begin{proof}[Proof of Proposition \ref{prop:surgery}]
From Proposition \ref{prop:equivalence} and Lemma \ref{lm:zig zag} we see that 
$\leg_{l,r}$ is $F_{i,j}$ from Section \ref{sec:construction}, (where $j=l-r$ and 
$n = l+r+i+2$).  Thus adding a right handed Dehn twist to the open book along $\leg_{l,r}$ 
gives an open book compatible with $(Y_n, \eta_{i,j}^n)$ and this describes the surgery 
from $Y_\infty$ to $Y_n$.  

To find the second link $\mathcal{C}$ and prove the lemma, we add positive twists to the 
monodromy of $(Y_\infty, \xi_{i+1})$ in a small standard region of a compatible open book 
and see that, after applying the braid relation of Lemma \ref{lm:braid}, we have an open 
book compatible with $(Y_\infty, \xi_{i})$.  Notice now that each open book for $(Y_\infty, 
\xi_{i+1})$, $i\geq0$, has a region:
\begin{center} \includegraphics[height=1.2in]{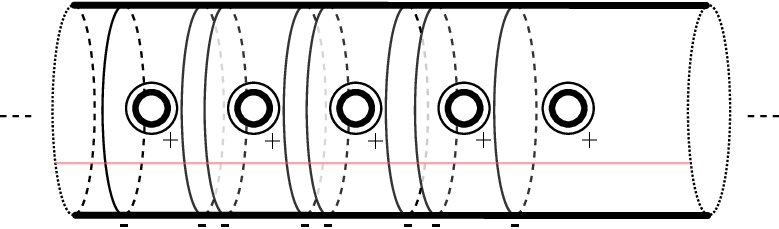}\end{center}
 which describes a region of Giroux torsion one plus a basic slice.  
To this, we add positive twists to the monodromy to get to the following open book. 
\begin{center} \includegraphics[height=1.2in]{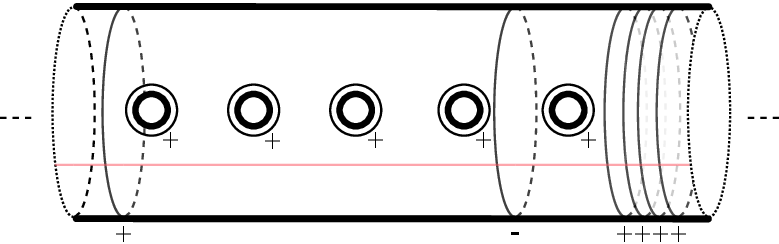}\end{center} 
These positive twists make up the Legendrian link $\mathcal{C}$.  Now repeatedly apply 
the braid relation and reduce to the open book for $(Y_\infty, \xi_i)$ where the Giroux 
torsion has been excised. Locally we now have the following picture.
\begin{center} \includegraphics[height=1.2in]{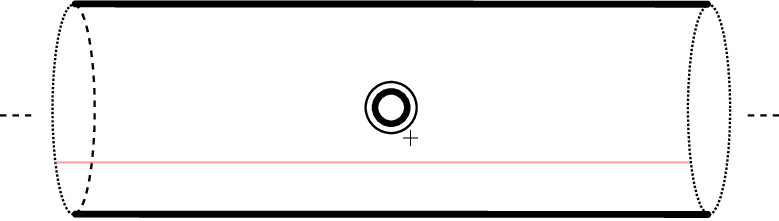} \end{center} 
Since the braid relation can be applied in the presence of Dehn twists along a horizontal 
curve (i.e. the red curve in the pictures above, see Figure \ref{fig:braid}), attaching 
Stein handles along 
$\leg_{l,r}\cup \mathcal{C}$ in $(Y_\infty, \xi_{i+1})$ gives an open book which is related 
to that shown in Figure \ref{fig:books} for $(Y_n, \eta_{i,j}^n)$ by Hopf stabilization.  
We showed in Lemma \ref{lm:braid} that the braid relation does not change how 
horizontal knots on the page are embedded up to an ambient contact isotopy, and thus 
after surgery and applying braid relation the image of the Legendrian knot $\leg_{l,r}$ 
from $\xi_{i+1}$ sits on a page of the open book as $\leg_{l,r}$ in $\xi_{i}$.
\end{proof}
%%%%%%%%%%%%%%%%%%%%%%%%%%%%%%%%%%%%%%%%%%%%%
\subsection{The computation of the invariants}
\begin{lemma} \label{homotopy1}
All tight contact structures $\xi_i$ on $Y_{\infty}$ are homotopic and have 
$3$-dimensional homotopy invariant $\theta(\xi_i)= - 4$
\end{lemma}
\begin{proof}
The homotopy between all the contact structures $\xi_i$ was proved in 
\cite[Proposition 2]{giroux:3}. Therefore we can 
compute the $3$-dimensional homotopy invariant of $\xi_0$, which has a  Stein filling 
$(V_{\infty}, J_{\infty})$ obtained by attaching a Stein handle on a Legendrian right-handed 
trefoil knot with Thurston--Bennequin invariant $0$
(see Figure \ref{topsurg.fig}). It is easy to see that $c_1(J_{\infty})=0$, $\chi(V_{\infty})=2$, 
and $\sigma(V_{\infty})=0$. The formula for the $3$-dimensional homotopy 
invariant in \cite[Definition 4.2]{gompf:1} is
$$ \theta(\xi_0)= c_1(J_\infty)^2 - 2 \chi(V_\infty) -3 \sigma(V_\infty),$$
so $\theta (\xi_0)=-4$.
\end{proof}
\begin{lemma}[{\cite[Theorem 3.12]{ghiggini:3}}] \label{homotopy2}
All tight contact structures $\eta_{i,j}^n$ are homotopic and have 
$3$-dimensional homotopy invariant $\theta (\xi_i)= - 6$.
\end{lemma}
The reference \cite{ghiggini:3} computes $\frac{\theta}{4}$ for $i=0$ or $i=n$, but the 
proof can be extended to all cases without modification.

\begin{lemma} \label{hf1}
$\underline{\hf}(-Y_{\infty}) \cong \Z[H^1( -Y_{\infty})] \oplus \Z[H^1( -Y_{\infty})]$ with one 
summand in degree $\frac 12$ and one in degree $\frac 32$. Moreover 
$c(\xi_0)$ is a generator of the summand in degree $\frac 12$.
\end{lemma}
\begin{proof}
$-Y_{\infty}$ can be obtained by $0$-surgery on the left-handed trefoil knot and the 
Poincar\'e  sphere $\Sigma(2,3,5)$ can be obtained by $(-1)$-surgery on the same knot. 
Then the surgery exact triangle of Theorem \ref{twisted-surgery} gives:
\[ \xymatrix{ \hf(S^3)[t,t^{-1}]  \ar[rr]^{F} & & \hf(\Sigma(2,3,5))[t,t^{-1}]
\ar[dl] \\ & \underline{\hf}(-Y_{\infty}) \ar[ul] & } \]
where the horizontal 
map $F$ is induced by a cobordism $W$ constructed by the attachment of a 
$2$-handle with framing $-1$ along the left-handed trefoil knot.
Therefore the integer homology group $H_2(W)$ is generated
by a surface $\widehat{\Sigma}$ with self-intersection $\widehat{\Sigma}^2=-1$.
 The ${\rm Spin}^c$-structures on $W$ are indexed by integers $k$ such
that $\langle c_1(\mathfrak{s}_k), [\widehat{\Sigma}] \rangle = 2k+1$, so
$c_1(\mathfrak{s}_k)^2= -(2k+1)^2$. By Theorem \ref{twisted-surgery}, $F=
\sum \limits_{k \in \Z} F_{W, \mathfrak{s}_k} \otimes t^k$. For any $\spc$-structure 
$\mathfrak{s}_k$ the map $F_{W, \mathfrak{s}_k}$ shifts the degree by 
$$\frac 14 \left (c_1(\mathfrak{s}_{k})^2 -2 \chi(W) -3 \sigma(W) \right ) = -k(k+1).$$ 
Since $-k(k+1) \le 0$,  $\hf(S^3) \cong \Z_{(0)}$ and $\hf(\Sigma(2,3,5)) \cong \Z_{(2)}$, 
the horizontal map is trivial. This implies that $\underline{\hf}(-Y_{\infty}) \cong 
\Z[H^1( -Y_{\infty})] \oplus \Z[H^1( -Y_{\infty})]$. The second homology group of 
$Y_{\infty}$ is generated by an embedded torus, and
therefore the adjunction inequality \cite[Theorem 7.1]{O-Sz:2}, which holds also for
Heegaard Floer homology with twisted coefficients, implies that $\underline{\hf}(-Y_{\infty})$
is concentrated in the trivial $\spc$-structure.  

The Heegaard Floer homology groups 
for $\spc$-structures with torsion first Chern class admit an absolute $\Q$-grading 
\cite[Section 7]{O-Sz:3}, which we are now going to determine for 
$\underline{\hf}(-Y_{\infty})$.
The map $\hf(\Sigma(2,3,5))[t,t^{-1}] \to  \underline{\hf}(- Y_{\infty})$ is induced by a 
$2$-handle attachment. The first Betti number of $\Sigma(2,3,5)$ is smaller than the first 
Betti number of $Y_\infty$, so the handle is attached along a null-homologous knot with 
framing $0$. Then the induced map has degree $- \frac 12$. 

The map $ \underline{\hf}(-Y_{\infty}) \to \hf(S^3)[t,t^{-1}]$  is induced by a $2$-handle 
attachment along a homologically nontrivial knot, so it also has degree 
$- \frac 12$;
see for example \cite[Lemma 3.1]{O-Sz:4}. This implies that 
$$\underline{\hf}(-Y_{\infty}) \cong \Z[H^1( -Y_{\infty})] \oplus \Z[H^1( -Y_{\infty})],$$ 
with one summand in degree $\frac 32$ and the other one in degree 
$\frac 12$. 

The contact invariant $c(\xi_0)$ has degree $- \frac{\theta(\xi_0)}{4} - \frac 12 = \frac 12$
and is a generator of $\underline{\hf}_{1/2}(-Y_{\infty})$ by \cite[Theorem 4.2]{O-Sz:genus}.
\end{proof}

\begin{lemma}[{\cite[Section 3]{plam:1}}] \label{hf2}\label{basis}
For any $n \geq 2$, $\hf_{+1} (-Y_n) \cong \Z^{n-1}$ and the contact invariants 
$c(\eta_{0,-n+2}^n), \ldots, c(\eta_{0,n-2}^n)$ form a basis.
\end{lemma}

The surgery described in Proposition \ref{prop:surgery} produces a cobordism 
$Z_n$ from $Y_{\infty}$ to $Y_n$ which can be decomposed in two different ways:
\begin{itemize}
\item either as a cobordism $W_{\infty}$ from $Y_{\infty}$ to itself followed by a cobordism
$V_{n}$ from $Y_{\infty}$ to $Y_{n}$ if we attach $2$-handles along ${\mathcal C}$
first, and then along ${\mathcal L}$, 
\item or as a cobordisms $V_{n+1}$ from 
$Y_{\infty}$ to $Y_{n+1}$ followed by a cobordism $W_n$ from $Y_{n+1}$ to $Y_n$ if we 
attach $2$-handles along  ${\mathcal L}$ first, and then along ${\mathcal C}$.
\end{itemize}

These cobordisms induce 
maps on Heegaard Floer homology according to Theorem \ref{map-twisted}.
Now we compute the change of the coefficient group for the maps induced by the 
cobordisms above. Let $K(V_n) = \ker \left ( H^2(V_n, Y_{\infty})  \to H^2(V_n) \right )$.  
By the cohomology exact sequence the map $H^1(Y_{\infty}) \to H^2(V_n, Y_{\infty})$ is an 
isomorphism. Therefore $K(V_n) \cong H^1(Y_{\infty})$ and we can identify
$\Z[K(V_n)]$ with $\Z[H^1(Y_{\infty})]$. Here we have used the fact that $Y_n$ is 
an integer homology sphere.
% which, in turn, will be identified with the Laurent 
% polynomial ring $\Z[t,t^{-1}]$.
The manifolds $Y_n$ are integer homology spheres, so the groups
$K(W_n)$ are trivial. Moreover $H^2(Z_n,\partial Z_n)= H^2(W_n, \partial W_n) 
\oplus H^2(V_{n+1}, \partial V_{n+1})$ and $H^2(Z_n)= H^2(W_n) \oplus H^2(V_{n+1})$, 
so $K(Z_n) \cong K(W_n) \oplus K(V_{n+1}) \cong K(V_{n+1})$. 

\begin{lemma}\label{trivialconnecting}
The connecting homomorphism $\delta \colon H^1(Y_{\infty}) \to H^2(Z_n)$ in the 
Mayer--Vietoris exact sequence for the decomposition $Z_n=W_{\infty} \cup_{Y_{\infty}} V_n$
is the trivial map.
\end{lemma}
\begin{proof}
It is  easier to see this by taking the Poincar\'e--Lefschetz duals and looking at 
the associated map in the Mayer--Vietoris exact sequence for relative homology.  There 
$\delta \colon H^1(Y_{\infty}) \to H^2(Z_n)$ becomes 
$i \colon H_2(Y_{\infty}) \to H_2(Z_n, \partial Z_n)$.  This map is trivial because the torus 
generating $H_2(Y_{\infty})$ is homologous (indeed isotopic) to the torus generating
the second homology group of the copy of $Y_{\infty}$ in 
the boundary of $Z_n$.  This can be seen by examining $W_{\infty}$, which is built from 
$Y_{\infty}$ by adding 2-handles along curves, each lying on a torus fiber.  Hence the torus 
fibers in each boundary component of $W_\infty$ are isotopic.
\end{proof}

Lemma \ref{trivialconnecting} and Lemma \ref{Ksequence} imply that $K(V_n)(W_{\infty})
\cong K(Z_n)$, so also $\Z[K(W_{\infty})] \cong \Z[H^1(Y_{\infty})]$. Moreover the cobordism
$W_{\infty}$ satisfies the hypotheses of Lemma \ref{simplifymap}. 
 By Theorem \ref{map-twisted} and Lemma \ref{simplifymap}, $V_n$ and $W_{\infty}$ 
induce maps
\[ F_{V_n, \mathfrak{k}} \colon \hf(-Y_n) \to \underline{\hf}(-Y_{\infty}) \quad
 \text{and} \quad  F_{W_{\infty}, \mathfrak{k}} \colon \underline{\hf}(-Y_{\infty}) 
\to \underline{\hf}(-Y_{\infty}), \]
where the map  $F_{W_{\infty}, \mathfrak{k}}$ is $\Z[H^1(Y_\infty)]$-linear.
By an abuse of notation, we will denote the canonical $\spc$-structure on a symplectic 
cobordism by $\mathfrak{k}$ regardless of what the cobordism or the symplectic form 
are; in fact this will not be very important in our proof.
These maps fit into a diagram:

\begin{equation}\label{diagramma}
\xymatrix{
\hf(-Y_n) \ar^{F_{W_n, \mathfrak{k}}}[rr] \ar^{F_{V_n, \mathfrak{k}}}[d] & 
&\hf(-Y_{n+1}) \ar^{F_{V_{n+1}, \mathfrak{k}}}[d] \\
\underline{\hf}(-Y_{\infty}) \ar^{F_{W_{\infty}, \mathfrak{k}}}[rr] &  &
\underline{\hf}(-Y_{\infty})
}
\end{equation}
which commutes for a suitable choice of the maps in their equivalence class, because 
Lemma \ref{trivialconnecting} implies that the restriction map $\mathfrak{s} \mapsto 
(\mathfrak{s}|_{W_{\infty}}, \mathfrak{s}|_{V_n})$ gives an isomorphism 
$\spc(Z_n) \cong \spc(W_{\infty}) \times \spc(V_n)$, and we have a similar isomorphism
 $\spc(Z_n) \cong \spc(V_{n+1}) \times \spc(W_n)$ because $H^1(Y_n)=0$. 

From now on we will  make a change in the coefficient ring which will allow us to write our
formulas in a more symmetric form. Let $\Lambda= \Z[H^1(Y_{\infty}, \frac 12 \Z)]$ with
the $\Z[H^1(Y_{\infty})]$-module structure defined by the inclusion $H^1(Y_{\infty}) \subset
H^1(Y_{\infty}, \frac 12 \Z)$. Since $\Lambda$ is a free module over $\Z[H^1(Y_{\infty})]$, we
have 
\[ \underline{\hf}(- Y_{\infty}; \Lambda) \cong \underline{\hf}(- Y_{\infty}) 
\otimes_{\Z[H^1(- Y_{\infty})]} \Lambda \cong \Lambda_{\left (\frac 12 \right )} \oplus 
\Lambda_{\left ( \frac 32 \right )}. \]
We choose an identification $\underline{\hf}_{\frac 12}(- Y_{\infty}; \Lambda) \cong 
\Lambda$ such that $[c(\xi_0)]$ corresponds to $[1]$.
% From now on,  we will write $\underline{\hf}(- Y_{\infty})$ to denote
% $\underline{\hf}(- Y_{\infty}, \Lambda)$.

\begin{lemma} \label{pluto} We can choose a representative of $F_{V_n}$,
an identification of  $\underline{\hf}_{(\frac 12)}(-Y_{\infty}; \Lambda)$ with 
$\Z[t^{\pm \frac 12}]$, and signs for the contact invariants $c(\eta_{0,j})$ such that 
 \[ F_{V_n, \mathfrak{k}}(c(\eta_{0,j}^n))= t^{j/2}.\] 
\end{lemma}

\begin{proof}
Let us view the $4$-manifold $V_{\infty}$ used in the proof of Lemma \ref{homotopy1}, 
constructed by adding a 
$2$-handle to $D^4$ along the right-handed trefoil knot in Figure \ref{topsurg.fig} with
attaching framing $0$, as a cobordism from $S^3$ to $Y_{\infty}$  and let 
$X_n= V_{\infty} \cup_{Y_{\infty}} V_n$. The second homology group of $X_n$ is generated 
by the class $T$ of a torus fiber  in $Y_{\infty}$ and by the class of a sphere $S$ such that 
$S \cdot T=1$. 

Let $\mathfrak{s}_j$ be the $\spc$-structure on $X_n$ such that
$$\langle c_1(\mathfrak{s}_j), T \rangle =0 \quad \text{and} \quad 
\langle c_1(\mathfrak{s}_j), S \rangle =j \text{ with } j \equiv n  (\text{ mod } 2).$$ 
The restriction of all the $\spc$-structures $\mathfrak{s}_j$ to $V_\infty$ and $V_n$ 
coincide, so there is a generator $h$ of $H^1(Y_{\infty})$ such that 
$\mathfrak{s}_{j+2}= \mathfrak{s}_j + \delta(h)$.  

We denote $\mathfrak{s}_j|_{V_\infty}= \mathfrak{h}$, 
$\mathfrak{s}_j|_{V_n}= \mathfrak{k}$ (in fact one can verify that the $\spc$-structure 
$\mathfrak{s}_j|_{V_n}$ coincide with the $\spc$-structure in Diagram (\ref{diagramma})) 
and identify $\Lambda$ with $\Z[t^{\pm 1/2}]$ by sending $e^h$ to $t$.
By the composition formula in Theorem \ref{twisted-composition} we can choose 
$F_{V_n,\mathfrak{k}}$ 
such that
\[F_{X_n, \mathfrak{s}_j} = \Pi \circ F_{V_{\infty}, \mathfrak{h}} \circ t^{-j/2} \cdot 
F_{V_n, \mathfrak{k}}. \]
In fact we choose $F_{V_n,\mathfrak{k}}$ such that $F_{X_n, \mathfrak{s}_{-n}} = \Pi \circ 
F_{V_{\infty}, \mathfrak{h}} \circ t^{n/2} \cdot F_{V_n, \mathfrak{k}}$ and the formula follows from
Equation (\ref{eqn: twisted composition}) because $ \mathfrak{s}_j = \mathfrak{s}_{-n} +
\frac{j+n}{2} \delta(h)$.

The restriction map $H^2(V_{\infty}; \Z) \to H^2(Y_{\infty}; \Z)$ is an
isomorphism and $H^1(V_{\infty})=0$, so $K(V_{\infty}) \cong H^1(Y_{\infty})$, and then
$V_{\infty}$ induces an anti-$\Lambda$-linear map 
\[ F_{V_{\infty}} \colon \underline{\hf}(-Y_{\infty}; \Lambda) \to 
\hf(S^3)[t^{1/2},t^{-1/2}]. \]

Since the right-handed trefoil knot has a Legendrian representative with 
Thurston--Bennequin invariant $+1$, $V_{\infty}$ can be endowed with a Stein 
structure providing a Stein cobordism from $(S^3, \xi_{st})$ to $(Y_{\infty}, \xi_0)$
with canonical $\spc$-structure $\mathfrak{h}$,
so $[F_{V_{\infty}, \mathfrak{h}}(c(\xi_0))]= [c(\xi_{st})] $. Then, after identifying
both $\underline{\hf}_{\frac 12}(-Y_{\infty}; \Lambda)$ and $\hf(S^3)[t^{1/2},t^{-1/2}]$ with 
$\Z[t^{\pm 1/2}]$, we can choose $F_{V_{\infty},\mathfrak{h}}$ to be the conjugation map 
$t \mapsto t^{-1}$.

% Let $X_n= V_{\infty} \cup_{Y_{\infty}} V_n$ be 
% the cobordism from $S^3$ to $Y_n$ given by attaching $2$-handles to $D^4$ as 
% specified in Figure \ref{topsurg.fig}, then the composition formula gives, up
% to overall multiplication by a power of $t$,
% \[ F_{V_n; \omega} \circ F_{V_{\infty}; \omega} = 
% F_{X_n; \omega} = \sum_{\mathfrak{s} \in {\rm Spin}^c(X_n)}
%  t^{\mathfrak{s}-\mathfrak{s}_0}F_{X_n, \mathfrak{s}} \]
% for some arbitrary reference ${\rm Spin}^c$-structure $\mathfrak{s}_0 \in 
% {\rm Spin}^c(X_n)$. 

The $\spc$-structure $\mathfrak{s}_j$ on $X_n$ is the canonical 
${\rm Spin}^c$-structure of the Stein filling $(X_n, J_n)$ of $\eta_{0,j}^n$ described 
by the Legendrian surgery diagram in Figure \ref{legsurg.fig}. Then we know from 
\cite{plam:1} that $F_{X_n, \mathfrak{s}_j} (c(\eta_{0,j}^m))= c(\xi_{st})$, and 
$F_{X_n, \mathfrak{s}_k} (c(\eta_{0,j}^n))= 0$ for $k \ne j$.
 Using the composition formula in Theorem \ref{twisted-composition} and the fact 
that $F_{V_{\infty}, \mathfrak{h}}$ is, in our choice of identifications, the 
conjugation map, we conclude that
$F_{V_n, \mathfrak{k}}(c(\eta_{0,j}^n))=t^{j/2}$. 
\end{proof}

\begin{figure}[ht] \centering
\psfrag{a}{\footnotesize $\frac{n-j}{2}$ cusps}
\psfrag{b}{\footnotesize $\frac{n+j}{2}$ cusps}
\includegraphics[width=6cm]{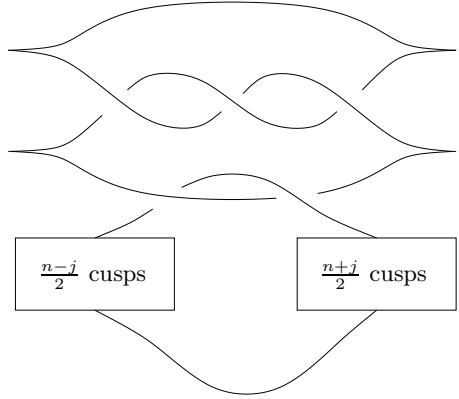}
\caption{Legendrian surgery presentation of $(Y_n, \eta_{0,j}^n)$.}
\label{legsurg.fig}
\end{figure}

Now we choose the maps in Diagram (\ref{diagramma}) so that it becomes commutative.
The horizontal map in the upper part  is fixed because the 
$Y_n$ are integer homology spheres, while the vertical maps are fixed by the choices in 
Lemma \ref{pluto}.
\begin{lemma}
If we choose $F_{V_n, \mathfrak{k}}$  such that 
$F_{V_n,  \mathfrak{k}}(c(\eta_{0,j}^n))= t^{j/2}$ for all $n$, then Diagram
(\ref{diagramma}) commutes  if we choose the map 
$F_{W_{\infty}, \mathfrak{k}}$  to be represented by the 
multiplication by $t^{\frac 12}-t^{- \frac 12}$. 
\end{lemma}
\begin{proof}
The contact structure $\xi_1$ is obtained from $\xi_0$ by a generalized Lutz twist,  
so $[c(\xi_1)]= [(t-1) c(x_0)]$ by \cite[Theorem 2]{ghiggini-honda:1}. This implies
that $F_{W_{\infty},  \mathfrak{k}}$ is the multiplication by $(t-1)t^{k/2}$ for some $k \in \Z$.
 We will now determine which choice for $F_{W_{\infty}, \mathfrak{k}}$
will make Diagram (\ref{diagramma}) commutative.

Inverting the orientation of the contact planes results in a symmetry of the triangle 
(\ref{pascal}) about its vertical axis. In particular the contact structure $\eta_{n-2,0}^n$ is 
invariant under conjugation, and $\eta_{0,j}^n$ is conjugated to $\eta_{0,-j}^n$;
see \cite[Proposition 3.8]{ghiggini:4}, where $\eta_{n-2,0}^n$ is called $\eta_0$, and 
$\eta_{0,j}^n$ is called $\eta_j$. In the reference only odd $n$ are considered, but the 
proof carries through in general. By Lemma \ref{basis} we know that $c(\eta_{n-2,0}^n)$ 
can be expressed as a linear combination of the elements $c(\eta_{0,i}^n)$. The invariance of
$c(\eta_{n-2,0}^n)$ by conjugation implies that 
$c(\eta_{n-2,0}^n)= a_{1-n} c(\eta_{0,1-n}^n) + \ldots + a_{n-1} c(\eta_{0,n-1}^n)$ with 
$a_j = a_{-j}$. hence  $F_{V_n, \mathfrak{k}}(c(\eta_{n-2,0}^n))$ is a symmetric Laurent 
polynomial in the variable $t^{1/2}$ because it is invariant under the automorphism 
$t^{1/2} \mapsto t^{-1/2}$. 

Since $F_{V_{n+1}, \mathfrak{k}}(c(\eta_{n-2,0}^n))= c(\eta_{n-1,0}^{n+1})$, the composite
$F_{V_{n+1}, \mathfrak{k}} \circ F_{W_n,\mathfrak{k}}$ maps $c(\eta_{n-2,0}^n)$ to a symmetric
Laurent polynomial. If Diagram (\ref{diagramma}) commutes, then 
$F_{W_{\infty}, \mathfrak{k}}$ maps symmetric Laurent polynomials to symmetric Laurent 
polynomials and therefore it must be the multiplication by $t^{1/2}-t^{- 1/2}$.
\end{proof}

\begin{proof}[Proof of Theorem \ref{main}]
The theorem will be proved by induction on $n$. The initial step is $n=2$. Since 
 there is a unique tight contact structure on $Y_2$ by 
\cite[Theorem 4.9]{ghiggini-schonenberger},  
there is nothing to prove in this case. 
Now we assume that Formula (\ref{f:main}) holds for the tight contact structures on 
$Y_n$, for some $n$, and we prove that this implies that Formula (\ref{f:main}) holds for 
the tight contact structures on $Y_{n+1}$. From the surgery construction we have
\[ F_{W_n,  \mathfrak{k}}(c(\eta_{i,j}^n)) = c(\eta_{i+1, j}^{n+1}), \]
and the induction hypothesis gives, on $Y_{n+1}$,
the following expression for the contact invariants of $\eta_{i,j}^{n+1}$ for $i \ge 1$ in 
terms of the contact invariants of $\eta_{1,j}^{n+1}$:
\begin{equation} \label{induction}
  c(\eta_{i+1,j}^{n+1})= \sum \limits_{k=0}^i (-1)^k  \binom{i}{k} c(\eta_{1,j-i+2k}^{n+1}). 
\end{equation}

We can compute $c(\eta_{1,j}^{n+1})$ by the the commutativity of Diagram 
\ref{diagramma}: in fact
\[ F_{V_{n+1}, \mathfrak{k}}(c(\eta_{1,j}^{n+1}))= 
F_{W_{\infty}, \mathfrak{k}}(F_{V_n,  \mathfrak{k}}(c(\eta_{0,j}^n))) =
t^{(j+1)/2} - t^{(j-1)/2}, \]
and therefore 
\begin{equation} \label{secondmain}
c(\eta_{1,j}^{n+1})= c(\eta_{0,j+1}^{n+1})-c(\eta_{0,j-1}^{n+1})
\end{equation}
because the map $F_{V_{n+1}, \mathfrak{k}}$ is injective.

If we substitute $c(\eta_{1,j}^{n+1})$ in Equation \ref{induction} with the right-hand side of 
Equation \ref{secondmain}, and write $j-i+1+2k= j-(i+1)+2(k+1)$, we obtain:

\begin{align*}
& c(\eta_{i+1,j}^{n+1}) = \sum \limits_{k=0}^{i} (-1)^k \binom{i}{k} ( 
c(\eta_{0,j-(i+1)+2(k+1)}^{n+1}) - c(\eta_{0,j-(i+1)+2k}^{n+1})) \\
& = \sum \limits_{k=1}^{i+1} (-1)^{k-1} \binom{i}{k-1} c(\eta_{0,j-(i+1)+2k}^{n+1}) - 
\sum \limits_{k=0}^{i} (-1)^k \binom{i}{k} c(\eta_{0,j-(i+1)+2k}^{n+1})\\
& = \sum \limits_{k=0}^{i+1} (-1)^{k-1} \left [ \binom{i}{k} + \binom{i}{k-1} \right ] 
c(\eta_{0,j-(i+1)+2k}^{n+1}) \\
& = \sum \limits_{k=0}^{i+1} (-1)^{k-1} \binom{i+1}{k} c(\eta_{0,j-(i+1)+2k})
\end{align*}
from which the statement of Theorem \ref{main} follows.
\end{proof}

% \bibliographystyle{plain}
% \bibliography{contatto}

\begin{thebibliography}{} 
\bibitem{dehn38} 
M.~Dehn.
\newblock Die {G}ruppe der {A}bbildungsklassen.
\newblock {\em Acta Math.}, 69(1):135--206, 1938.

\bibitem{etnyre-honda:1}
J.~Etnyre and K.~Honda.
\newblock On the nonexistence of tight contact structures.
\newblock {\em Ann. of Math. (2)}, 153(3):749--766, 2001.

\bibitem{ghiggini:nonstein}
P.~Ghiggini.
\newblock Strongly fillable contact 3-manifolds without {S}tein fillings.
\newblock {\em Geom. Topol.}, 9:1677--1687, 2005.

\bibitem{ghiggini:4}
P.~Ghiggini.
\newblock Infinitely many universally tight contact manifolds with trivial
  {O}zsv\'ath-{S}zab\'o contact invariants.
\newblock {\em Geom. Topol.}, 10:335--357, 2006.

\bibitem{ghiggini:3}
P.~Ghiggini.
\newblock Ozsv\'ath-{S}zab\'o invariants and fillability of contact structures.
\newblock {\em Math. Z.}, 253(1):159--175, 2006.

\bibitem{ghiggini-honda:1}
P.~Ghiggini and K.~Honda.
\newblock Giroux torsion and twisted coefficients.
\newblock arXiv:0804.1568.

\bibitem{ghiggini-honda-vhm}
P.~Ghiggini, K.~Honda and J.~Van Horn-Morris.
\newblock The vanishing of the contact invariant in the presence of torsion.
\newblock arXiv:0706.1602

\bibitem{ghiggini-schonenberger}
P.~Ghiggini and S.~Sch{\"o}nenberger.
\newblock On the classification of tight contact structures.
\newblock In {\em Topology and Geometry of Manifolds}, volume~71 of {\em
  Proceedings of Symposia in Pure Mathematics}, pages 121--151. American
  Mathematical Society, 2003.

\bibitem{giroux:3}
E.~Giroux.
\newblock Une infinit{\'e} de structures de contact tendues sur une
  infinit{\'e} de vari{\'e}t{\'e}s.
\newblock {\em Invent. Math.}, 135(3):789--802, 1999.

\bibitem{giroux}
E.~Giroux.
\newblock Structures de contact en dimension trois et bifurcations des feuilletages de 
surfaces.
\newblock {\em Invent. Math.}, 141(3):615--689, 2000.

\bibitem{gompf:1}
R.~Gompf.
\newblock Handlebody construction of {S}tein surfaces.
\newblock {\em Ann. of Math. (2)}, 148(2):619--693, 1998.

\bibitem{ga}
M.~Greenberg and J.~Harper.
\newblock Algebraic topology. A first course. 
\newblock Mathematics Lecture Note Series, 58. {\em Benjamin/Cummings Publishing 
Co., Inc., Advanced Book Program, Reading, Mass.} 1981.

\bibitem{honda:1}
K.~Honda.
\newblock On the classification of tight contact structures {I}.
\newblock {\em Geom. Topol.}, 4:309--368, 2000.

\bibitem{honda:2}
K.~Honda.
\newblock On the classification of tight contact structures {II}.
\newblock {\em J. Differential Geom.},  55(1):83--143, 2000.


\bibitem{jabuka-mark}
S.~Jabuka and T.~Mark.
\newblock Product formulae for {O}zsv\'ath-{S}zab\'o 4-manifold invariants.
\newblock {\em Geom. Topol.}, 12(3):1557--1651, 2008.

\bibitem{johns79}
D.~Johnson.
\newblock Homeomorphisms of a surface which act trivially on homology.
\newblock {\em Proc. Amer. Math. Soc.} 75(1):119--125, 1979.

\bibitem{lisca-matic:1}
P.~Lisca and G.~Mati{\'c}.
\newblock Tight contact structures and {S}eiberg-{W}itten invariants.
\newblock {\em Invent. Math.}, 129(3):509--525, 1997.

\bibitem{lisca-stipsicz:4}
P.~Lisca and A.~Stipsicz.
\newblock Ozsv\'ath-{S}zab\'o invariants and tight contact three-manifolds.
  {I}.
\newblock {\em Geom. Topol.}, 8:925--945, 2004.

\bibitem{Onaran}
S.~Onaran.
\newblock {\em Legendrian knots and open book decompositions}.
\newblock {Ph. D.} Thesis, Middle East Technical University, 2009.

\bibitem{O-Sz:4}
P.~Ozsv{\'a}th and Z.~Szab{\'o}.
\newblock Absolutely graded {F}loer homologies and intersection forms for
  four-manifolds with boundary.
\newblock {\em Adv. Math.}, 173(2):179--261, 2003.

\bibitem{O-Sz:genus}
P.~Ozsv{\'a}th and Z.~Szab{\'o}.
\newblock Holomorphic disks and genus bounds.
\newblock {\em Geom. Topol.}, 8:311--334, 2004.

\bibitem{O-Sz:2}
P.~Ozsv{\'a}th and Z.~Szab{\'o}.
\newblock Holomorphic disks and three-manifold invariants: properties and
  applications.
\newblock {\em Ann. of Math. (2)}, 159(3):1159--1245, 2004.

\bibitem{O-Sz:cont}
P.~Ozsv{\'a}th and Z.~Szab{\'o}.
\newblock Heegaard {F}loer homology and contact structures.
\newblock {\em Duke Math. J.}, 129(1):39--61, 2005.

\bibitem{O-Sz:3}
P.~Ozsv{\'a}th and Z.~Szab{\'o}.
\newblock Holomorphic triangles and invariants for smooth four-manifolds.
\newblock  {\em Adv. Math.}, 202(2):326--400, 2006

\bibitem{plam:1}
O.~Plamenevskaya.
\newblock Contact structures with distinct {H}eegaard {F}loer invariants.
\newblock {\em Math. Res. Lett.}, 11(4):547--561, 2004.

\bibitem{vanhorn}
J.~Van Horn-Morris.
\newblock {\em Contructions of open book decompositions}.
\newblock {Ph. D.} dissertation, University of Texas at Austin, 2007.

\bibitem{wu:surgery}
H.~Wu.
\newblock On {L}egendrian surgeries.
\newblock {\em Math. Res. Lett.}, 14(3):513--530, 2007.

\end{thebibliography}
% %\include{paper.bbl}

\end{document}